\documentclass[12pt]{amsart}
\usepackage{amsmath,amsthm,amsfonts,amssymb,bm,bbm,graphicx,comment}
\usepackage[margin=1in]{geometry}    
\usepackage{natbib}
\usepackage{color}
\usepackage{parskip}  
\usepackage{hyperref}
\usepackage{cleveref}
\usepackage[utf8]{inputenc}
\usepackage{todonotes}

\newtheorem{definition}{Definition}[section]
\newtheorem{theorem}{Theorem}[section]

\newtheorem{prop}{Proposition}[section]
\newtheorem{remark}{Remark}[section]
\newtheorem{assumption}{Assumption}[section]
\numberwithin{equation}{section}

\newcommand{\R}{\mathbb{R}}

\newcommand{\N}{\mathbb{N}}

\newcommand{\E}{\mathbb{E}}

\title[Edge-indexed network time series]{Edge-indexed network time series with graph Ornstein-Uhlenbeck dynamics}
\author{Jiaming Chen}
\address{Department of Mathematics, Imperial College London, UK}
\email{j.chen1@imperial.ac.uk}
\author{Almut E. D. Veraart}
\address{Department of Mathematics, Imperial College London, UK}
\email{a.veraart@imperial.ac.uk}
\begin{document}

\maketitle

\begin{abstract}
We introduce a class of L\'evy-driven graph Ornstein-Uhlenbeck (grOU) models for edge-indexed network time series. The proposed framework extends generalized network autoregressive (GNAR) processes for edge-indexed network time series to continuous time and adapts graph Ornstein-Uhlenbeck dynamics, originally developed for node-indexed processes, to the edge-indexed setting. The model accommodates general L\'evy noise and therefore captures both Brownian and jump behavior. We show that the model parameters can be estimated via a maximum-likelihood framework and derive the asymptotic properties of the estimator. We examine the finite-sample performance of the methodology through simulation studies and illustrate its practical relevance in an empirical application to high-frequency financial data. The results indicate that grOU models for edge-indexed network time series improve forecasting accuracy and reduce computational time relative to standard benchmarks while maintaining robustness through their network-based parametrization.
\end{abstract}


\section{Introduction}
Multivariate time series arise in many applications, including finance, economics, medical science and environmental science. In such settings, the dimension of the time series is typically large, and this requires models that are both flexible and parsimonious. Classical approaches often struggle to meet these competing demands. For instance, in vector autoregressive processes, the number of parameters grows quadratically with the dimension. This quickly becomes prohibitive in high-dimensional settings. 

To address these challenges, some recent research has focused on network-based models for multivariate time series, or more generally, network stochastic processes. In this paradigm, the dependence structure is encoded in the topology of a graph, allowing models to exploit sparsity in the system. The network autoregressive (NAR) model introduced by \cite{10.1214/16-AOS1476} and its generalization, the generalized network autoregressive (GNAR) model of \cite{knight2016modelling} and \cite{knight2020generalized}, illustrate how information flows across network nodes and their neighborhoods can be incorporated. Variants of these models, particularly the edge-based formulations introduced in \cite{mantziou2023gnar}, have since been proposed, extending the methodology from a new perspective.

Building on these discrete-time formulations, there is a growing interest in continuous-time analogs that naturally handle irregular sampling intervals and jump dynamics in high-frequency data. In particular, L\'evy-driven Ornstein-Uhlenbeck (OU) processes (\cite{barndorff2001non}) and their network extensions, such as graph OU processes (\cite{courgeau2022high}, \cite{courgeau2022likelihood}) and  multivariate continuous autoregressive (MCAR) processes (\cite{lucchese2026estimation}) provide a natural continuous-time framework with tractable dependence structures. It is therefore natural to extend GNAR-type models, and in particular the GNAR edge model, to their continuous-time counterparts, leading to the grOU framework for edge-indexed time series. Such models inherit the flexibility of OU dynamics while maintaining a network edge-based structure; accordingly, we aim to demonstrate their strong performance in explaining and forecasting correlated and high-frequency multivariate data. Although our simulation and empirical analyzes use a uniform time grid for simplicity and comparability with discrete-time benchmark models, the model itself is formulated in continuous time; inference depends only on the observed time increments and therefore remains valid under irregular sampling schemes.

Specifically, our work is motivated by a data set consisting of high-frequency time series of limit order books for a collection of U.S. equities and exchange-traded assets obtained from the LOBSTER database during the period 2023–2026. The sample spans multiple economic sectors and includes the SPDR S\&P 500 ETF (SPY) together with NASDAQ, Inc. (NDAQ). These assets naturally form a moderately sized financial network in which potential connections arise from sectoral and index membership.

We first evaluate the finite-sample properties of the proposed estimation procedure through a simulation study. The methodology is then applied to the high-frequency limit order book data to estimate and forecast the dynamics of edge processes under alternative network specifications, with predictive performance assessed against standard benchmarks. To facilitate reproducibility, the paper is accompanied by an open-source software release available at \href{https://github.com/jiamingchen13/GraphOUEdge}{https://github.com/jiamingchen13/GraphOUEdge}, where all code required to replicate the simulation and empirical analyzes is provided.

The rest of the article is organized as follows. Section \ref{Preliminaries and Model Setup} reviews multivariate continuous-time autoregressive processes and introduces the grOU process for edge-indexed time series. We also present concrete parametric specifications based on lag structures and neighbor-stage interactions. Section \ref{Parameter Estimation} develops the estimation framework for grOU edge-indexed processes. More precisely, we construct estimators for the drift parameters using discrete approximations to the stochastic integrals and time derivatives arising in the continuous-time formulation. Section \ref{Simulation study} examines the finite-sample performance of the estimators through a Monte Carlo simulation study. Section \ref{Empirical illustration} provides an empirical application using high-frequency limit order book data for a network of U.S. equities and exchange-traded assets, where the forecasting performance of the grOU edge model is compared with several benchmark approaches. Section \ref{Conclusion and Outlook} concludes with a discussion of the main findings and directions for future research, while some proofs of the principal theoretical results are collected in the Appendix.

\section{Graph Ornstein-Uhlenbeck models for edge-indexed network time series}
\label{Preliminaries and Model Setup}
\subsection{Notations}
We work on a filtered probability space $(\Omega,\mathcal{F},(\mathcal{F}_t,t\geq 0),\mathbb{P})$, to which all stochastic processes are assumed to be adapted. For $n,d\in\N$, let $\mathcal{M}_{n,d}(\mathbb{R})$ denote the space of $n\times d$ real matrices; when $n=d$, we drop the redundant subscript and write $\mathcal{M}_{d}(\mathbb{R})$. The $d\times d$ identity matrix and zero matrix are denoted by $\mathbf{I}_{d}$ and $\mathbf{0}_{d}$, respectively. For matrices $A,B$, the notation $A\odot B$ refers to their Hadamard product. For a set $S$, $\mathbbm{1}_{S}$ stands for its indicator function. The $(i,j)$-th entry of a matrix $A\in\mathcal{M}_{n,d}(\mathbb{R})$ is written as $A(i,j)$, and its transpose as $A^{\top}$. The vectors of ones and zeros, with a dimension clear from the context, are denoted by $\mathbf{1}$ and $\mathbf{0}$. For scalars $a_1,...,a_n$, the diagonal matrix with diagonal entries $(a_1,...,a_n)$ is denoted by $\text{diag}(a_1,...,a_n)$. The $n$-th component of a vector $\mathbf{a}$ is written as $\mathbf{a}^{(n)}$. We denote by $\mathbb{S}_d^{+}$ and $\mathbb{S}_d^{++}$ the cones of $d\times d$ symmetric, positive semi-definite and positive definite matrices, respectively. The norm of vectors or matrices is denoted by $\Vert\cdot\Vert$. Since all norms on finite-dimensional spaces are equivalent, the particular choice is immaterial. The spectrum of a matrix is denoted by $\sigma(\cdot)$ with further defining
\begin{equation*}
    M_d^-:=\{A\in\mathcal{M}_{d}(\mathbb{R})\vert\sigma(A)\subset(-\infty,0)+i\R\},
\end{equation*}
the set of $d\times d$ matrices whose eigenvalues have strictly negative real parts. Such matrices are also referred to as Hurwitz matrices.

For the convergence of random variables, $\overset{\mathbb{P}}{\longrightarrow}$ denotes convergence in probability, while $\overset{D}{\longrightarrow}$ denotes convergence in distribution.

\subsection{Networks and edge neighbors} 
In the following, we are interested in the setup where we observe time series on the edges of a network or graph. Let $\mathcal{G}=(V,E)$ be a static and known graph with node set $V=\{1,...,n\}$ and edge set $E\subseteq V \times V$. For a directed graph, $(i,j)\in E$ denotes an edge from $i$ to $j$. In the undirected case, edges are symmetric, i.e., $(i,j)\in E\Leftrightarrow (j,i)\in E$.

Following the definition of $r$-stage neighbors in \cite{knight2020generalized} and the edge-based notation of \cite{mantziou2023gnar}, we define the notion of $r$-stage neighboring edges for a given edge $(i,j)\in E$ and $r\in\mathbb{N}_0:=\N\cup\{0\}$. By convention, the $0$-stage neighbor of an edge is the edge itself, i.e., $\mathcal{N}^{(0)}((i,j)) :=\{(i,j)\}$. For $r=1$, the set of neighboring edges consists of all edges incident to at least one of the nodes $i$ or $j$, excluding $(i,j)$ itself:
\begin{equation*}
\mathcal{N}^{(1)}((i,j)) = \{(k,l)\in\delta^+(\{i,j\})\cup\delta^-(\{i,j\})\vert(k,l)\neq(i,j)\},
\end{equation*}
where $\delta^+(\cdot)$ and $\delta^-(\cdot)$ denote the sets of outgoing and incoming edges of a set of nodes, respectively. For further stage $r\geq 2$, higher-order neighbors are defined recursively by iterating through the first-stage neighborhood and excluding all edges that have already been assigned to previous stages:
\begin{equation*}
\mathcal{N}^{(r)}((i,j)) = \mathcal{N}^{(1)}(\mathcal{N}^{(r-1)}((i,j)))\setminus[\cup_{q=0}^{r-1}\mathcal{N}^{(q)}((i,j))].
\end{equation*}
We illustrate this construction in Figure \ref{fig:edge-neighbors3} for the case of a directed network with 8 nodes.
\begin{figure}[htbp]
    \centering
    \includegraphics[width=0.7\textwidth, height=0.23\textheight]{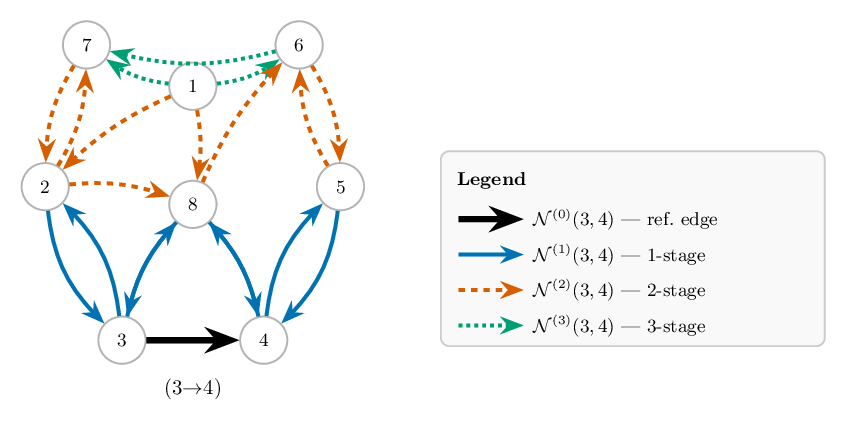}
    \caption{Illustration of the concept of neighbors of edges in a directed network with 8 nodes: Edge $(3,4)$ is the reference edge, and we show its 1st-stage (solid, blue), 2nd-stage (dashed, orange) and 3rd-stage (dotted, green) neighbors.}
    \label{fig:edge-neighbors3}
\end{figure}

Let $K:=\vert E\vert$ denote the total number of edges in the network and label them by $e_1,\dots,e_K$. Let $R_l$ the maximum neighborhood stage at lag $l$. For $r=1\dots,R_l$, define the weighted neighborhood matrices 
\begin{equation}\label{weight_matrix}
    \mathbf{W}^{(r)} := \begin{bmatrix}
    w_{e_1e_1}\mathbbm{1}\{e_1\in\mathcal{N}^{(r)}(e_1)\} & \dots  & w_{e_1e_K}\mathbbm{1}\{e_K\in\mathcal{N}^{(r)}(e_1)\} \\
    \vdots & \ddots & \vdots  \\
    w_{e_Ke_1}\mathbbm{1}\{e_1\in\mathcal{N}^{(r)}(e_K)\} & \dots  & w_{e_Ke_K}\mathbbm{1}\{e_K\in\mathcal{N}^{(r)}(e_K)
\end{bmatrix} \in \mathbb{R}^{K \times K}.
\end{equation}
The weights $w_{e_ie_j}\in[0,1]$ quantify the strength of the interaction between edges $e_i,e_j$. They may depend on the cardinality of the neighborhood, structural properties of the graph, or empirical measures of association between edge processes. We assume the normalization condition $\sum_{e_j\in\mathcal{N}^{(r)}(e_i)}w_{e_ie_j}=1$ for all $e_i\in E$ and $r\in\N$ such that $\mathcal{N}^{(r)}(e_i)$ is non-empty. 

\begin{remark}
It is straightforward to extend the above framework to allow for multiple covariate-dependent weight matrices. We exclude such extensions here and focus on a single weighting scheme to simplify the exposition.
\end{remark}

The GNAR edge model introduced in \cite{mantziou2023gnar} assumes that the future evolution of the process observed on a given edge depends not only on its own past, but also on lagged values of processes observed on neighboring edges, yielding the following vector form
\begin{equation*}
   \mathbf{X}_t = \sum_{l=1}^L\left(\text{diag}(\alpha_{e_1,l},...,\alpha_{e_K,l})\mathbf{X}_{t-l} + \sum_{r=1}^{R_l}\beta_{l,r}\mathbf{W}^{(r)}\mathbf{X}_{t-l}\right)+\mathrm{noise}_t,
\end{equation*}
where $\alpha_{e_i,l}$ is the ``standard" autoregressive coefficient associated with edge $e_i$ at lag $l$, and $\beta_{l,r}$ measures the contribution of the $r$-th neighborhood order at lag $l$. For simplicity, we consider uniform weights, $w_{e_ie_j}=\vert \mathcal{N}^{(r)}(e_i)\vert^{-1}$, whenever $e_j\in\mathcal{N}^{(r)}(e_i)$.

\subsection{A continuous-time modeling framework}
To facilitate a continuous-time formulation, we rewrite the coefficient matrix $\mathbf{Q}_l \in \mathbb{R}^{K \times K}$, for each lag $l = 1,\dots,L$, as
\begin{equation*}
\mathbf{Q}_l := \text{diag}(\alpha_{e_1,l},...,\alpha_{e_K,l})+ \sum_{r=1}^{R_l}\beta_{l,r}\textbf{W}^{(r)}. 
\end{equation*}
With $\bm{\alpha}_l=(\alpha_{e_1,l},\dots,\alpha_{e_K,l})$ and $\bm{\beta}_l = (\beta_{l,1},\dots,\beta_{l,R_l})$, the full parameter vector is then defined by
\begin{equation}
\label{theta}
    \theta := [\bm{\alpha}_1,\bm{\beta}_1,\dots,\bm{\alpha}_L,\bm{\beta}_L]^{\top}\in\R^{LK+ \sum_{l=1}^LR_l}.
\end{equation}

Let $\mathbf{L}_t$ be a $K$-dimensional L\'evy process with L\'evy-Khintchine triplet $(\mathbf{b},\Sigma, \nu)$, where $\mathbf{b}\in\R^K$, $\Sigma\in\mathbb{S}^{+}_K$, and $\nu$ is a L\'evy measure on $\R^K$ satisfying $\int_{\R^K}(1\wedge\Vert\mathbf{z}\Vert^2)\nu(d\mathbf{z})<\infty$. The characteristic function of $\mathbf{L}_t$ is then given, for $\mathbf{u}\in\R^K$ at time $t\geq 0$, by
\begin{equation*}
\E[\exp(i\mathbf{u}^{\top}\mathbf{L}_t)]=\exp\left\lbrace t\left(i\mathbf{u}^{\top}\mathbf{b}-\frac{1}{2}\mathbf{u}^{\top}\Sigma\mathbf{u}+\int_{\R^K\setminus\{0\}}(\exp(i\mathbf{u}^{\top}\mathbf{z})-1-i\mathbf{u}^{\top}\mathbf{z}\mathbbm{1}_{\Vert\mathbf{z}\Vert<1})\nu(d\mathbf{z})\right)\right\rbrace.
\end{equation*}

We are now in the position to define the graph Ornstein-Uhlenbeck process for edge-indexed network time series.
\begin{definition}[Graph Ornstein-Uhlenbeck processes for edge-indexed network time series]
\label{groudef}
    Let $(\mathbf{L}_t)_{t\geq 0}$ be a $K$-dimensional L\'evy process with L\'evy-Khintchine triplet $(\mathbf{b},\Sigma, \nu)$. Following the state-space construction in \cite[Lemma 3.8]{marquardt2007multivariate}, we define the matrices $\mathbb{Q}=\mathbb{Q}(\theta)$ and $\mathcal{E}$ as
\begin{equation*}
    \mathbb{Q}:=\begin{bmatrix}
    \mathbf{0}_{K} & \mathbf{I}_{K} & \mathbf{0}_{K} & \dots & \mathbf{0}_{K}\\
    \mathbf{0}_{K} & \mathbf{0}_{K} & \mathbf{I}_{K}  & \dots & \mathbf{0}_{K}\\
    \vdots & \vdots & \vdots & \ddots & \vdots  \\
    \mathbf{0}_{K} & \mathbf{0}_{K} & \mathbf{0}_{K} & \dots & \mathbf{I}_{K}\\
    -\mathbf{Q}_L & -\mathbf{Q}_{L-1} & -\mathbf{Q}_{L-2} & \dots  & -\mathbf{Q}_1
\end{bmatrix}\in \mathcal{M}_{LK}(\mathbb{R}),~\mathcal{E}:=\begin{bmatrix}
    \mathbf{0}_{K}\\
    \mathbf{0}_{K}\\
    \vdots\\
    \mathbf{0}_{K}\\
    \mathbf{I}_{K}
\end{bmatrix}\in \mathcal{M}_{LK,K}(\mathbb{R}).
\end{equation*}
The graph Ornstein-Uhlenbeck process $\mathbf{Y}_t$ for edge-indexed network time series with maximum lag $L$ and underlying network $\mathcal{G}$ is defined by
\begin{equation*}
\mathbf{Y}_t := \mathbf{A}\mathbf{X}_t\text{ with }\mathbf{A}=[\mathbf{I}_{K},\mathbf{0}_{K},\dots,\mathbf{0}_{K}]\in\mathcal{M}_{K,LK}(\mathbb{R}),
\end{equation*}
where the state vector $\mathbf{X}_t\in\R^{LK}$ satisfies the linear stochastic differential equation
\begin{equation}
\label{model2}
d\mathbf{X}_t = \mathbb{Q}\mathbf{X}_tdt + \mathcal{E}d\mathbf{L}_t,~\mathbf{X}_0=\xi, ~t \geq 0.
\end{equation}
\end{definition}

We abbreviate the above model as grOU$(L,[R_1,\dots,R_L])$, where $L$ denotes the maximum lag in the autoregression and $R_l$ specifies the maximum neighborhood stage at lag $l$. Formally, the process $\{\mathbf{Y}_t,t\geq0\}$ is defined through the differential relation
\begin{equation*}
    D^L\mathbf{Y}_t+\mathbf{Q}_1D^{L-1}\mathbf{Y}_t+ \dots+\mathbf{Q}_{L-1}D\mathbf{Y}_t+\mathbf{Q}_L\mathbf{Y}_t = D\mathbf{L}_t,
\end{equation*}
where the operator $D$ denotes differentiation with respect to $t$. Since $D\mathbf{L}_t$ is, in general, not well-defined as a random function, the above equation is understood in a formal sense via an equivalent $\mathbb{R}^{LK}$-valued state-space representation $\{\mathbf{X}_t,t\geq0\}$. In particular, \cite[Proposition 2.1]{lucchese2026estimation} shows that this representation can be recovered from a realization of $\{\mathbf{Y}_t,t\geq0\}$ through right differentiation. Specifically, 
we write $\mathbf{X}_s=(\mathbf{X}_{s,1}^{\top},\ldots,\mathbf{X}_{s,L}^{\top})^{\top}$, where $\mathbf{X}_{s,l}$ denotes the $l$th $K$-dimensional block corresponding to lag $(l-1)$, for $l=1,\ldots, L$. Then, for $l=1,\dots,L$, the $(l-1)$-fold derivative of the observed process satisfies
\begin{equation*}
    D^{l-1}\mathbf{Y}_{s} = \mathbf{X}_{s,l}.
\end{equation*}

\begin{remark}\label{rem:stationay}
By \cite[Remark 2.1]{lucchese2026estimation}, suppose that $\xi$ is independent of the natural filtration generated by $\mathbf{L}_t$ and has the representation $\xi = \int_{-\infty}^0e^{-s\mathbb{Q}}\mathcal{E}d\mathbf{L}_s$, where we are now working with a two-sided L\'evy process $(\mathbf{L}_t)_{t\in \mathbb{R}}$. If $\mathbb{Q}\in M_{LK}^-$ and the L\'evy measure $\nu$ satisfies the logarithmic moment condition: $\int_{\Vert\mathbf{z}\Vert>1}\log\Vert\mathbf{z}\Vert\nu(d\mathbf{z})<\infty$, then, by \cite[Theorem 4.1]{SATO198473}, equation \eqref{model2} admits a unique strictly stationary solution given by
\begin{equation*}
    \mathbf{X}_t=e^{t\mathbb{Q}}\xi+\int_0^te^{(t-s)\mathbb{Q}}\mathcal{E}d\mathbf{L}_s=\int_{-\infty}^te^{(t-s)\mathbb{Q}}\mathcal{E}d\mathbf{L}_s, ~~\text{for any }t\geq 0.
\end{equation*}
Moreover, by \cite[Proposition 3.34]{marquardt2007multivariate}, if the driving L\'evy process is $r$-integrable for some $r>0$, i.e., $\int_{\Vert \mathbf{z}\Vert\geq1}\Vert \mathbf{z}\Vert^r\nu(d\mathbf{z})<\infty$, then both $\mathbf{X}_t$ and $\mathbf{Y}_t$ are ergodic.
\end{remark}

We are now interested in characterizing the second-order moments of the graph Ornstein-Uhlenbeck process in the stationary setting (see Remark \ref{rem:stationay}). 
\begin{prop}\label{prop:secondorder} Using the notation introduced above, we set $\mathbf{Y}=(\mathbf{Y}_t)_{t\in \mathbb{R}}$, where $\mathbf{Y}_t=\mathbf{A}\mathbf{X}_t=\mathbf{A}\int_{-\infty}^t e^{(t-s)\mathbb{Q}}\mathcal{E}d\mathbf{L}_s$, and we make the assumption $\int_{\Vert\mathbf{z}\Vert>1}\Vert\mathbf{z}\Vert^2\nu(d\mathbf{z})<\infty$ and assume that
$\mathbb{Q}\in M_{LK}^-$, i.e., that $\mathbb{Q}$ is a Hurwitz matrix. We fix the following notation:
$\boldsymbol{\mu}_{\mathbf{L}}=\E(\mathbf{L}_1)=\mathbf{b}+\int_{\Vert\mathbf{z}\Vert>1}\mathbf{z}\nu(d\mathbf{z})$ and $\boldsymbol{\Sigma}_{\mathbf{L}}=\text{Var}(\mathbf{L}_1)=\Sigma +\int_{\mathbb{R}^K}\mathbf{z}\mathbf{z}^{\top} \nu(d\mathbf{z})$. We get the following results: $\mathbb{Q}$ and $\mathbf{Q}_L$ are invertible, and  
    \begin{align*}
  \mathbb{Q}^{-1}
 =
 \begin{bmatrix}
 -\mathbf{Q}_L^{-1}\mathbf{Q}_{L-1} & -\mathbf{Q}_L^{-1}\mathbf{Q}_{L-2} & \dots & -\mathbf{Q}_L^{-1}\mathbf{Q}_1 & -\mathbf{Q}_L^{-1}\\
 \mathbf{I}_{K} & \mathbf{0}_{K} & \dots & \mathbf{0}_{K} & \mathbf{0}_{K}\\
 \mathbf{0}_{K} & \mathbf{I}_{K} & \dots & \mathbf{0}_{K} & \mathbf{0}_{K}\\
 \vdots & \vdots & \ddots & \vdots & \vdots\\
 \mathbf{0}_{K} & \mathbf{0}_{K} & \dots & \mathbf{I}_{K} & \mathbf{0}_{K}
 \end{bmatrix},
 &&  \mathbb{Q}^{-1}\mathcal{E}=
 \begin{bmatrix}
 - \mathbf{Q}_{L}^{-1}\\
    \mathbf{0}_{K}\\
    \mathbf{0}_{K}\\
    \vdots\\
    \mathbf{0}_{K}   
\end{bmatrix}.
\end{align*}
Moreover, the unconditional mean, variance and autocovariance matrix of $\mathbf{Y}$ are given by
\begin{align*}
    \E[\mathbf{Y}_0] 
    &= \mathbf{A}\E(\mathbf{X}_0)=\mathbf{A}\int_{-\infty}^0e^{-s\mathbb{Q}} ds \mathcal{E}\boldsymbol{\mu}_{\mathbf{L}} = -\mathbf{A}\mathbb{Q}^{-1}\mathcal{E}\boldsymbol{\mu}_{\mathbf{L}}=\mathbf{Q}_L^{-1}\boldsymbol{\mu}_{\mathbf{L}},\\
    \mathrm{Var}[\mathbf{Y}_0]
     &=\mathbf{A}\mathrm{Var}[\mathbf{X}_0]\mathbf{A}^{\top}
     =\mathbf{A}\int_{-\infty}^0e^{-s\mathbb{Q}}\mathcal{E}\boldsymbol{\Sigma}_{\mathbf{L}}\mathcal{E}^{\top} e^{-s\mathbb{Q}^{\top}}ds  \mathbf{A}^{\top}   =\mathbf{A}\Gamma\mathbf{A}^{\top},\\
     \mathrm{Cov}(\mathbf{Y}_h, \mathbf{Y}_{0})
     &=\mathbf{A} e^{h\mathbb{Q}} \Gamma \mathbf{A}^{\top}, \quad h \geq 0,
\end{align*}
where $\Gamma$ is the unique solution to the Lyapunov equation $\mathbb{Q}\Gamma + \Gamma\mathbb{Q}^{\top} = -\mathcal{E}\boldsymbol{\Sigma}_{\mathbf{L}}\mathcal{E}^{\top}$. 
\end{prop}
\begin{proof} The proof of this proposition follows from straightforward calculations.
\end{proof}
\begin{remark}
We note that the formula for the stationary mean simplifies such that it only depends on the inverse of the coefficient matrix at the maximal lag $L$ and the mean of the driving L\'{e}vy process.
\end{remark}
Next, we derive the conditional mean and conditional covariance of $\mathbf{Y}_{t+h}$ given $\mathcal{F}_t$ since these quantities are relevant if one is interested in prediction.
\begin{prop}\label{prop:secondorder-cond} Under the assumption of Proposition \ref{prop:secondorder}, let $t\in \mathbb{R},\, h \geq 0$, then the conditional mean and variance satisfy
\begin{align*}
    \mathbb{E}[\textbf{Y}_{t+h}\vert\mathcal{F}_t] &=
    \E[\mathbf{Y}_0]+\mathbf{A}e^{h\mathbb{Q}}[\mathbf{X}_{t}-\E[\mathbf{X}_0]], \quad \mathrm{where}\,\, \E[\mathbf{X}_0]=-\mathbb{Q}^{-1} \mathcal{E}\boldsymbol{\mu}_{\mathbf{L}}, \\
    \mathrm{Var}[\textbf{Y}_{t+h}\vert\mathcal{F}_t] &=
    \mathbf{A} \left(\int_{0}^{h}e^{u\mathbb{Q}}\mathcal{E}\boldsymbol{\Sigma}_L\mathcal{E}^{\top} e^{u\mathbb{Q}^{\top}}du\right)\mathbf{A}^{\top}.
\end{align*}
\end{prop}
\begin{proof}
The result can be derived from the independent increments property of $\textbf{L}_t$, as detailed in Appendix \ref{prop_proof}.
\end{proof}

\section{Parameter estimation}
\label{Parameter Estimation}
The grOU edge model can be considered as a special case of the multivariate continuous-time autoregressive (MCAR) framework, where we incorporate a known and static network structure into the dynamics of the observed process. In this setting, the components of an MCAR process can be interpreted as edges of a graph endowed with a neighborhood structure. The connectivity between edges is encoded through the edge-based adjacency matrices introduced in Section \ref{Preliminaries and Model Setup}, which, in turn, determine the structure of the coefficient matrices defined in \eqref{weight_matrix}. Exploiting this graphical representation, we derive the explicit estimator for the model parameter vector $\theta$ defined in \eqref{theta} under both continuous and discrete observation regimes.

\begin{theorem}[\cite{lucchese2026estimation} Theorem 5.1]
\label{thm}
    Let $\{\mathbf{Y}_s,s\in[0,t]\}$ be observations of a stationary and ergodic $\mathrm{grOU}(L, [R_1,\dots,R_L])$ process with parameter
    \begin{equation*}
    \theta^*\in\Theta:=\left\lbrace\theta\in\R^{LK+\sum_{l=1}^LR_l}:\mathbb{Q}(\theta)\in M_K^-\right\rbrace,
    \end{equation*}
    and let $(\mathbf{W}^{(r)})_{r=1, \ldots, R_l}$ denote the collection of weighted neighbor matrices at lag $l$ defined in \eqref{weight_matrix}. Assume that the process is defined under $\mathbb{P}^{\theta^*}$ and driven by a square-integrable L\'evy process. Let $\mathbf{Q}^{(0)}=(\mathbf{0}_K,\dots,\mathbf{0}_K)$, and denote by $D^{L-1}\mathbf{Y}^c_{\mathbf{Q}^{(0)},t}$ the $(L-1)$-th lag component of the continuous local martingale part of $\mathbf{X}_t$ defined in \eqref{martingale}. Define
    \begin{equation*}
        \mathbf{H}_s := \begin{bmatrix}
    \mathrm{diag}\left(D^{L-1}\mathbf{Y}_{s-}\right)\\
    (\mathbf{W}^{(1)}D^{L-1}\mathbf{Y}_{s-})^{\top}\\
    \vdots\\
    (\mathbf{W}^{(R_1)}D^{L-1}\mathbf{Y}_{s-})^{\top}\\
    \vdots\\
    \mathrm{diag}\left(\mathbf{Y}_{s-}\right)\\
    (\mathbf{W}^{(1)}\mathbf{Y}_{s-})^{\top}\\
    \vdots\\
    (\mathbf{W}^{(R_L)}\mathbf{Y}_{s-})^{\top}\\
    \end{bmatrix}\in\mathcal{M}_{LK+\sum_{l=1}^LR_l,K}(\mathbb{R}).
    \end{equation*}
    Then the maximum likelihood estimator given by 
    \begin{equation*}
\theta_{MLE}^*\left(\mathbf{Y}_{[0,t]}\right)=[\mathbf{K}]_t^{-1}\mathbf{K}_t,
    \end{equation*}
    where
\begin{equation}
\label{K_t}
    \mathbf{K}_t=-\int_0^t\mathbf{H}_s\Sigma^{-1}dD^{L-1}\mathbf{Y}^{c}_{\mathbf{Q}^{(0)},s}
\text{ and }
    [\mathbf{K}]_t=\int_0^t\mathbf{H}_s\Sigma^{-1}\mathbf{H}_s^{\top}ds,
\end{equation}
is
\begin{itemize}
    \item (weakly) consistent, i.e., as $t\to\infty$,
    \begin{equation*}
\theta_{MLE}^*\left(\mathbf{Y}_{[0,t]}\right)\overset{\mathbb{P}^{\theta^*}}{\longrightarrow}{\theta}^*;
    \end{equation*}
    \item asymptotically normal, i.e., as $t\to\infty$,
    \begin{equation*}
\sqrt{t}\left(\theta_{MLE}^*\left(\mathbf{Y}_{[0,t]}\right)-\theta^*\right)\overset{D}{\longrightarrow}Z\sim \mathcal{N}\left(\mathbf{0},\mathbf{K}_{\infty}^{-1}\right),
    \end{equation*}
where $\mathbf{K}_{\infty}\in\mathcal{M}_{LK+\sum_{l=1}^LR_l}(\mathbb{R})$ is a symmetric positive definite matrix such that $t^{-1}[\mathbf{K}]_t\overset{\mathbb{P}^{\theta^*}}{\longrightarrow}\mathbf{K}_{\infty}, t\to\infty$.
\end{itemize}
\end{theorem}
\begin{proof}
    See Appendix \ref{thm_proof}.
\end{proof}

The continuous-time observation framework considered in Theorem \ref{thm} is, however, not realistic in most applications. In practice, data are observed on discrete (and possibly irregular) time grids. To accommodate this, we introduce a discretized estimation procedure. Assume that the process is observed on a discrete grid, and we replace the $l$-th derivative operator $D^l$ from the continuous-time formulation with the forward $l$-th finite-difference approximation, denoted by the operator $\hat{D}^l$. Convergence of higher-order finite differences requires a uniformity condition on the observation grid. The detailed construction of the discretization scheme and the associated assumptions are provided in Appendix \ref{discrete}.

The following result provides a discrete analog of the continuous-time estimator and corresponds to the discrete version of \cite[Theorem 5.1]{lucchese2026estimation}. This estimator is used throughout our simulation and empirical analyzes.

\begin{theorem}
\label{thm2}
Let $\mathbf{Y}_{\mathcal{P}_t}$ denote a discretely sampled version of $\{\mathbf{Y}_s,s\in[0,t]\}$, where $\{\mathcal{P}_t;t\geq 0\}$ is a countable sequence of partitions with refinements $\{\mathcal{Q}_t;t\geq 0\}$ and associated threshold sequences $\bm{\nu}^m_t$ satisfying the assumptions stated in Appendix \ref{discrete}. Define
\begin{equation*}
    \hat{\mathbf{H}}_s = \begin{bmatrix}
    \mathrm{diag}\left(\hat{D}^{L-1}\mathbf{Y}_{s}\right)\\
    (\mathbf{W}^{(1)}\hat{D}^{L-1}\mathbf{Y}_{s})^{\top}\\
    \vdots\\
    (\mathbf{W}^{(R_1)}\hat{D}^{L-1}\mathbf{Y}_{s})^{\top}\\
    \vdots\\
    \mathrm{diag}\left(\mathbf{Y}_{s}\right)\\
    (\mathbf{W}^{(1)}\mathbf{Y}_{s})^{\top}\\
    \vdots\\
    (\mathbf{W}^{(R_L)}\mathbf{Y}_{s})^{\top}\\
    \end{bmatrix}\in\mathcal{M}_{LK+\sum_{l=1}^LR_l,K}(\mathbb{R}).
\end{equation*}
Then the discretized estimator
    \begin{equation*}
    \theta_D^*\left(\mathbf{Y}_{\mathcal{P}_t};\mathcal{Q}_t,\bm{\nu}_t\right):= \left([\mathbf{K}]^{-1}_{\mathcal{P}_t,\mathcal{Q}_t}\mathbf{K}_{\mathcal{P}_t,\mathcal{Q}_t,\bm{\nu}_t}\right),     
    \end{equation*}
    where
    \begin{equation*}
\mathbf{K}_{\mathcal{P}_t,\mathcal{Q}_t,\bm{\nu}_t}=-\sum_{m=0}^{M_t-1}\hat{\mathbf{H}}_{u_m} \Sigma^{-1}\left[\Delta^m_{\mathcal{Q}_t}\hat{D}^{L-1}\mathbf{Y}-\mathbf{b}\Delta^m_{\mathcal{Q}_t}\right]\odot\mathbbm{1}_{\left\lbrace \vert \Delta^m_{\mathcal{Q}_t}\hat{D}^{L-1}\mathbf{Y}-\mathbf{b}\Delta^m_{\mathcal{Q}_t}\vert\leq\bm{\nu}^m_t\right\rbrace}
    \end{equation*} and
    \begin{equation*}
        [\mathbf{K}]_{\mathcal{P}_t,\mathcal{Q}_t}=\sum_{m=0}^{M_t-1}\hat{\mathbf{H}}_{u_m}\Sigma^{-1}\hat{\mathbf{H}}_{u_m}^{\top}(u_{m+1}-u_m),
    \end{equation*}
is
    \begin{itemize}
    \item (weakly) consistent, i.e., as $t\to\infty$,
    \begin{equation*}
    \theta_D^*\left(\mathbf{Y}_{\mathcal{P}_t};\mathcal{Q}_t,\bm{\nu}_t\right)\overset{\mathbb{P}^{\theta^*}}{\longrightarrow}\theta^*;
    \end{equation*}
    \item asymptotically normal, i.e., as $t\to\infty$,
    \begin{equation*}
        \sqrt{t}\left(\theta_{D}^*\left(\mathbf{Y}_{\mathcal{P}_t};\mathcal{Q}_t,\bm{\nu}_t\right)-\theta^*\right)\overset{D}{\longrightarrow}Z\sim \mathcal{N}\left(\mathbf{0},\hat{\mathbf{K}}^{-1}_{\infty}\right),
    \end{equation*}
    where $\hat{\mathbf{K}}_{\infty}\in\mathcal{M}_{LK+\sum_{l=1}^LR_l}(\mathbb{R})$ is a symmetric positive definite matrix such that $t^{-1}[\mathbf{K}]_{\mathcal{P}_t,\mathcal{Q}_t}\overset{\mathbb{P}^{\theta^*}}{\longrightarrow}\hat{\mathbf{K}}_{\infty}, t\to\infty$.
\end{itemize}
\end{theorem}
\begin{proof}
    See Appendix \ref{thm2_proof}.
\end{proof}

\section{Simulation study}
\label{Simulation study}
We carry out a simulation study to answer the following three questions: First, how well does the proposed estimation procedure perform in finite samples? Second, how well does the new model class perform in prediction tasks compared to other popular alternatives? Third, how robust is the estimation procedure towards model misspecification?

\subsection{Finite-sample behavior and parameter consistency} We begin by investigating the finite-sample properties of the estimators via a Monte Carlo study. In particular, we focus on assessing the consistency in finite samples by checking whether the estimators concentrate around the true parameter values as the observation horizon increases.

For simplicity, we consider a uniform time discretization. Following the framework of \cite[Section 4.4]{lucchese2026estimation}, we set the mesh sizes $\Delta_{\mathcal{P}_t} = t^{-6}$ and $\Delta_{\mathcal{Q}_t} = t^{-2}$, ensuring that the conditions of Theorem \ref{thm2} are satisfied. We consider time horizons $t = 2, 4, 8$ to illustrate convergence as $t$ increases.

We construct a network with three vertices and two undirected edges, incorporating two temporal lags, as shown below in Figure \ref{fig:edge-neighbors1}. Let $\alpha_{i,j}$ denote the ``standard" autoregressive parameter associated with edge $i$ at lag $j$. The true parameter matrix is given by
\begin{equation*}
    \alpha^* = \begin{bmatrix}
        4 & 2\\
        3 & 1\\
    \end{bmatrix} \text{, and } \beta^* = [1,1]^{\top},
\end{equation*}
where $\beta_i$ denotes the network power parameter at $i$.
\begin{figure}[htbp]
    \centering
    \includegraphics[scale = 1.2]{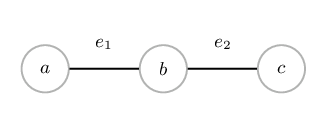}
    \caption{The undirected network with three vertices and two edges $e_1$ and $e_2$.}
    \label{fig:edge-neighbors1}
\end{figure}

The driving noise is a L\'evy process with a characteristic triplet $(\textbf{0},\textbf{I}_2,F)$. We consider three different jump regimes:
\begin{itemize}
    \item Pure diffusion, i.e., the process is driven by standard Brownian motion only.
    \item Compound Poisson process with common jump arrival times of intensity $\lambda =1$ and multivariate Gaussian jump sizes $\mathcal{N}(\textbf{0},\textbf{I}_2)$. 
    \item Independent symmetric Gamma-type process with shape $k=1$ and scale $\theta=1$.
\end{itemize}
We follow the exact procedure in \cite[Appendix G.2]{lucchese2026estimation} to simulate the first two regimes and use the approximate Euler–Maruyama scheme in \cite[Appendix G.1]{lucchese2026estimation} for the last one with infinitely many jumps.

In practice, the L\'evy triplet is unknown and must be estimated from data. We therefore implement the data-driven procedure described in \cite[Section 4.5]{lucchese2026estimation}. Due to limited high-frequency observations in our empirical dataset, the truncation threshold $\bm{\nu}_t$ mentioned in Theorem \ref{thm2} cannot be directly inferred from $\Delta \mathcal{Q}_t$. Instead, we adopt the method in \cite[Section H.1]{lucchese2026estimation} to disentangle the Brownian and jump components.

Figures \ref{sim_BM_fig}-\ref{sim_Gamma_fig} present the empirical distributions of the estimators based on $1000$ Monte Carlo replications for $t = 2, 4, 8$. Across all regimes, the estimators exhibit increasing concentration around the true parameters $\alpha^*$ and $\beta^*$ as $t$ grows, providing strong empirical support for the consistency results established in asymptotic theory. The model performs stably across all jump specifications. As is standard, small jumps are effectively absorbed into the continuous martingale component. In the remainder of the paper, we therefore focus on the compound Poisson setting.

\begin{figure}[htbp]
  \centering
  \includegraphics[scale = .6]{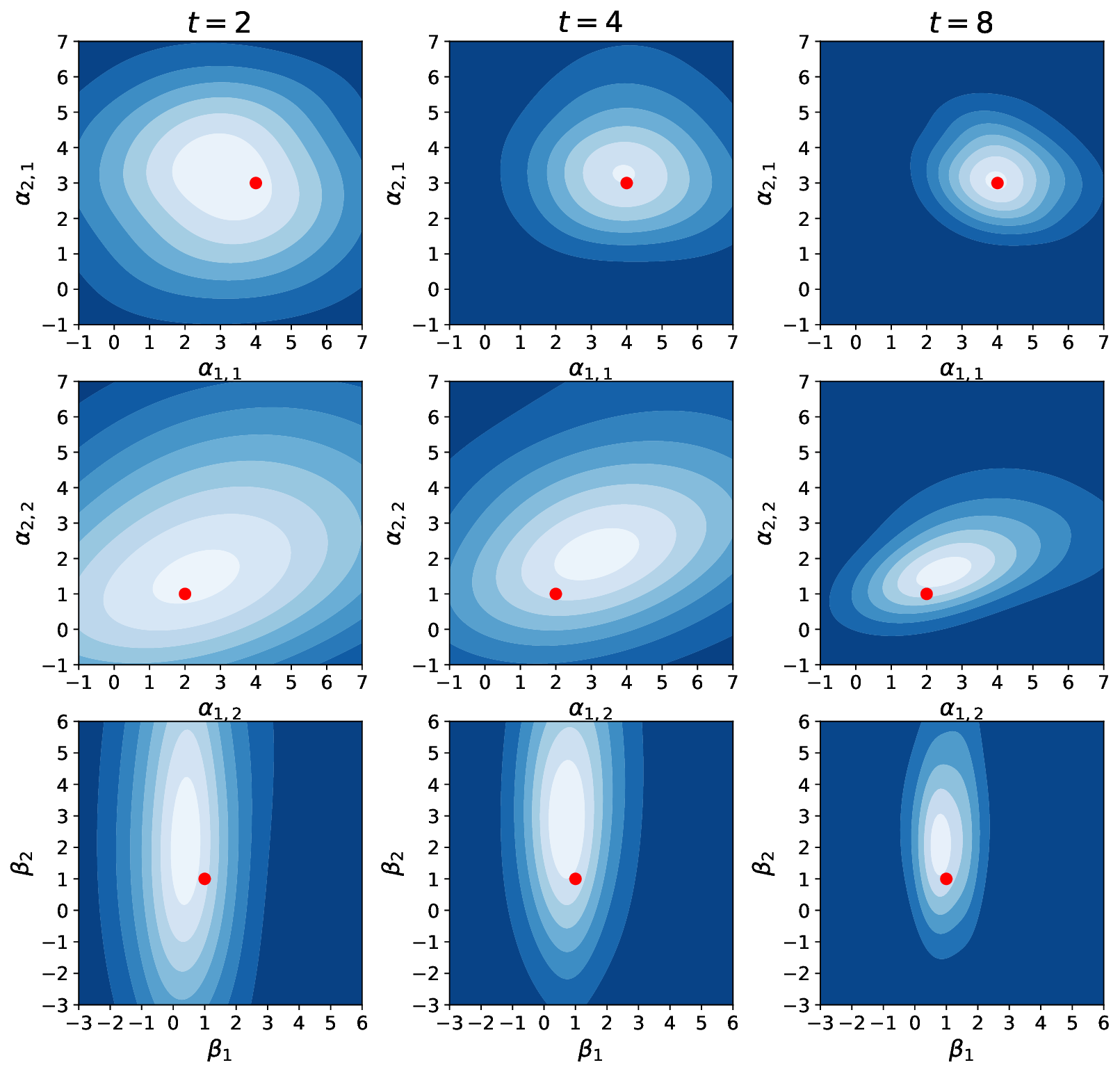}
  \caption{Empirical distributions of the ``standard'' autoregressive parameters and network power estimators under pure diffusion (Brownian motion) noise. Results are based on $1000$ Monte Carlo simulations for $t = 2, 4, 8$.}
  \label{sim_BM_fig}
\end{figure}

\begin{figure}[htbp]
  \centering
  \includegraphics[scale = .6]{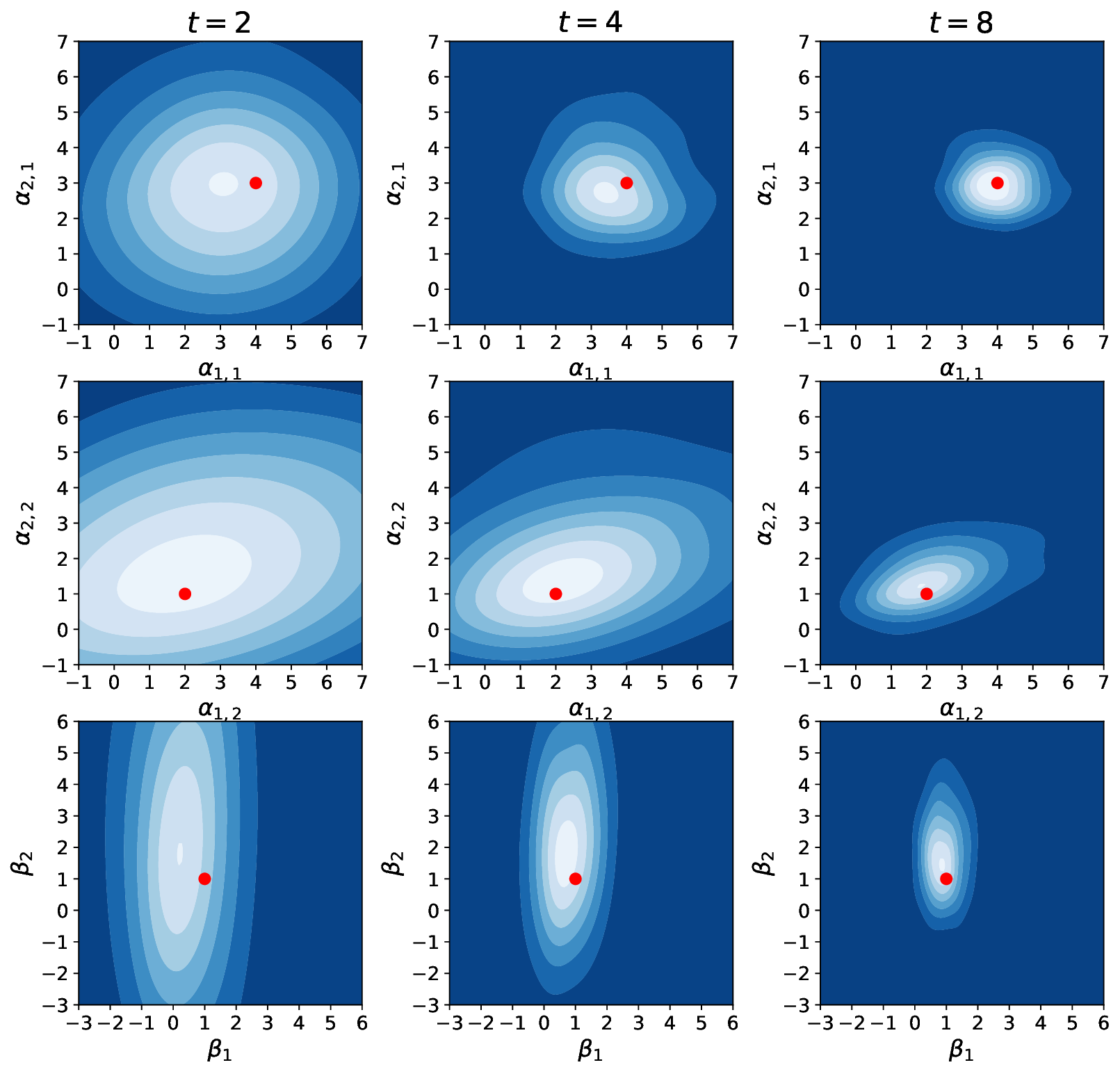}
  \caption{Empirical distributions of the ``standard'' autoregressive parameters and network power estimators under compound Poisson noise. Results are based on $1000$ Monte Carlo simulations for $t = 2, 4, 8$.}
  \label{sim_CP_fig}
\end{figure}

\begin{figure}[htbp]
  \centering
  \includegraphics[scale = .6]{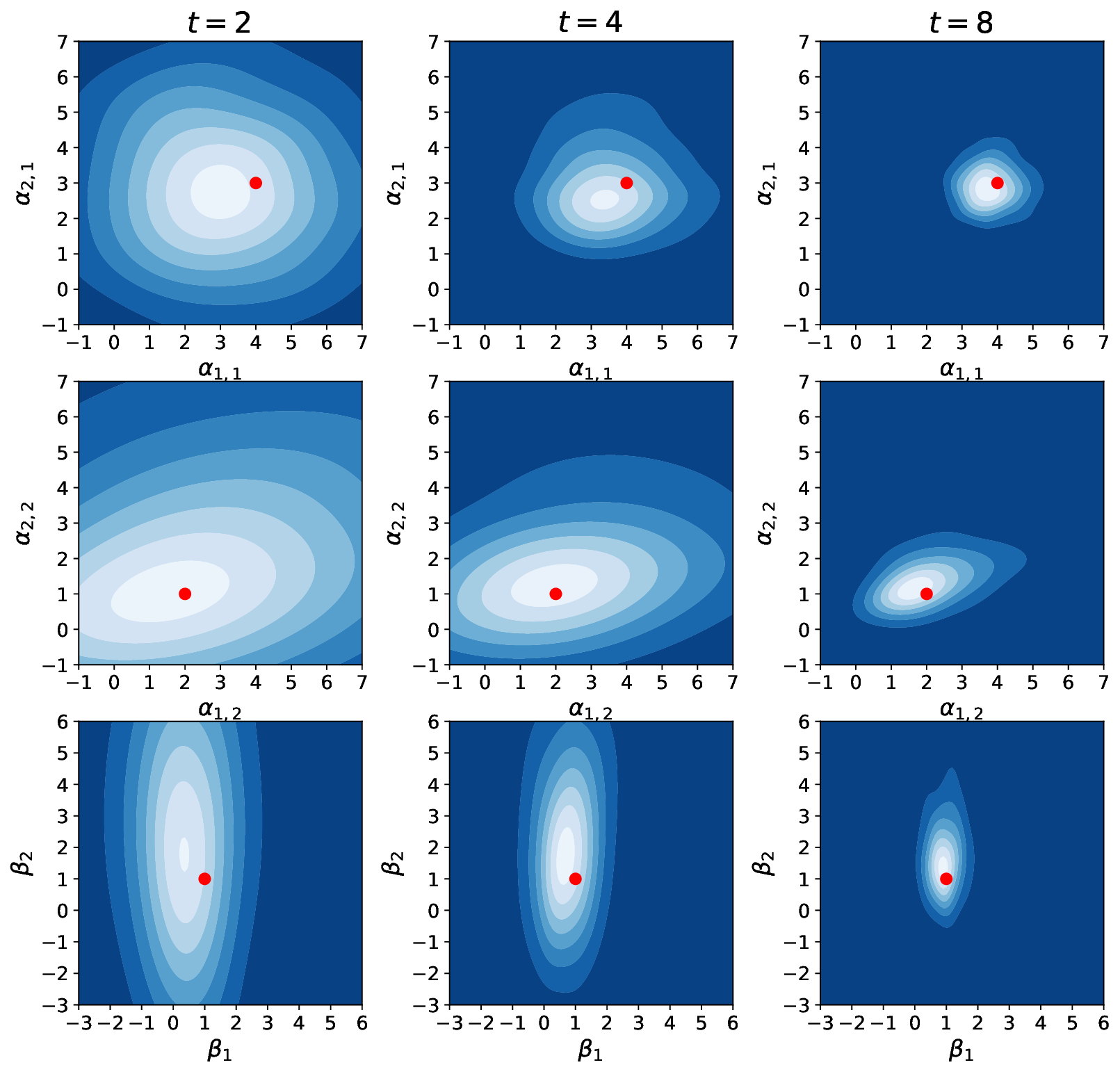}
  \caption{Empirical distributions of the ``standard'' autoregressive parameters and network power estimators under symmetric Gamma noise. Results are based on $1000$ Monte Carlo simulations for $t = 2, 4, 8$.}
  \label{sim_Gamma_fig}
\end{figure}

\subsection{Predictive performance}
\label{Predictive performance}
To assess predictive performance, we consider a complete graph with five vertices (ten edges), as shown in Figure \ref{fig:edge-neighbors2}. The ``standard'' autoregressive parameter is set to $\alpha = 5$ for one selected edge and $\alpha = 1$ for the remaining nine edges. In addition, we impose a strong network interaction effect of magnitude $2$ on each neighborhood stage for each lag.
\begin{figure}[htbp]
    \centering
    \includegraphics[scale = 1]{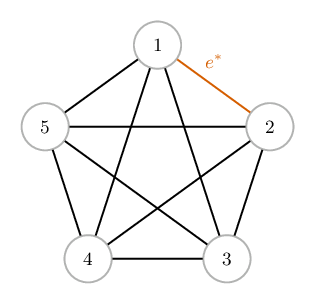}
    \caption{A complete graph on five vertices, in which a single distinguished edge $e^*$ is highlighted in color while all remaining edges are drawn uniformly.}
    \label{fig:edge-neighbors2}
\end{figure}

The system evolves according to a lag-one, first-neighborhood-stage grOU edge process driven by a compound Poisson process with intensity $\lambda = 1$ and multivariate Gaussian jump sizes $\mathcal{N}(\textbf{0},\sigma^2\textbf{I}_{10})$, where $\sigma^2 \in \{1, 5, 10\}$. This design introduces both substantial network dependence and increasing jump variability. A representative sample path for each level of jump-size variability is shown in Figure \ref{sample_path}.

\begin{figure}[htbp]
  \centering
  \includegraphics[scale = .35]{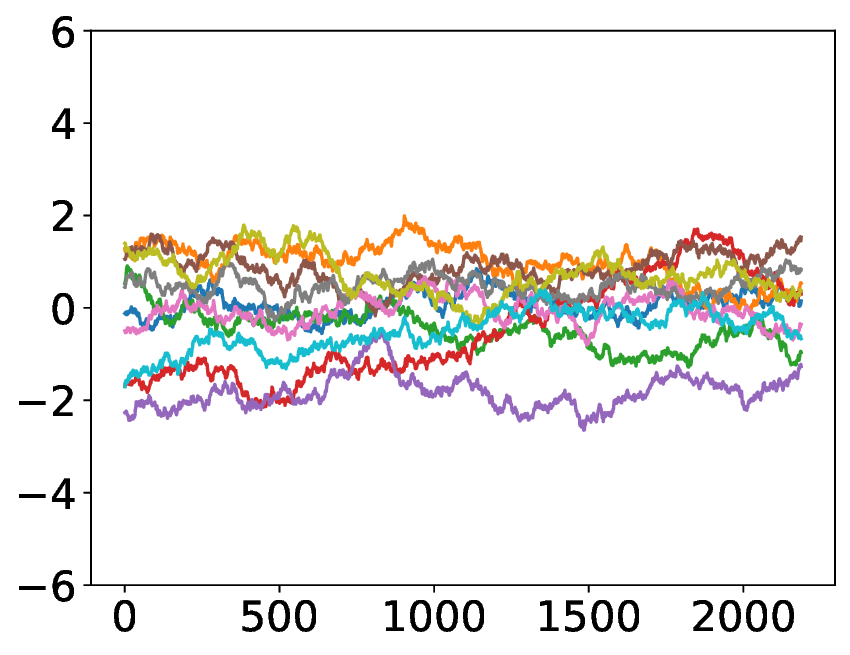}
  \includegraphics[scale = .35]{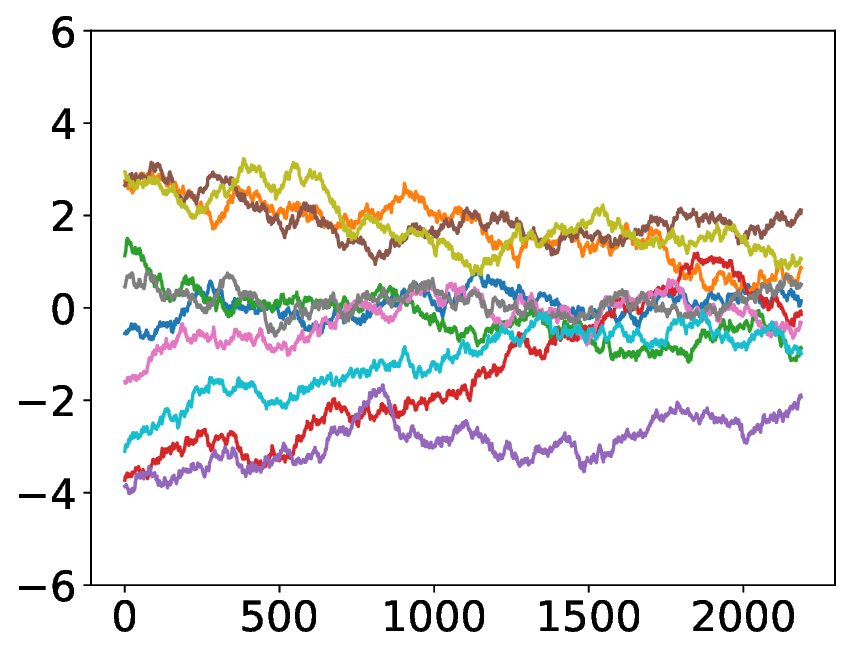}
  \includegraphics[scale = .35]{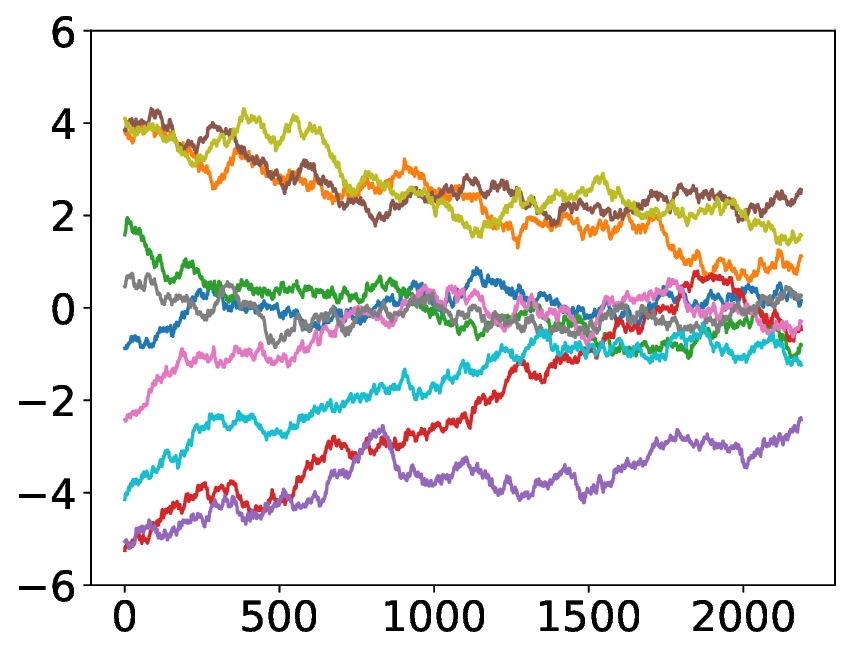}
  \caption{Sample paths of $10$ edges over the full observation period under jump-size variances $\sigma^2 = 1$ (left), $\sigma^2 = 5$ (middle), and $\sigma^2 = 10$ (right).}
  \label{sample_path}
\end{figure}

We make a one-step forecast and compare the grOU$(1, [1])$ model with the following benchmarks:
\begin{enumerate}
    \item\textbf{Naive model (NA):} We predict the next observation by the most recent observed value, that is, $\hat{\mathbf{Y}}_{t+1} = \mathbf{Y}_t$. This serves as a simple and widely used benchmark in financial applications. 
    \item \textbf{Univariate autoregressive edge model (AR):} Each edge is forecast independently using an AR(1) model. Assuming standard Gaussian noise, predictions are given by the conditional expectation based on the OLS estimator.
    \item \textbf{Vector autoregressive edge model (VAR):} All edges are jointly forecast under a VAR(1) model. Assuming multivariate standard Gaussian noise, predictions are obtained from the conditional expectation based on the OLS estimator.
    \item \textbf{Generalized network autoregressive edge model (GNAR):} All edges are forecast using a GNAR(1,[1]) edge model, following the methodology of \cite{mantziou2023gnar}.
    \item \textbf{Multivariate continuous autoregressive edge model (MCAR):} All edges are forecast using Proposition $\ref{prop:secondorder-cond}$, under a continuous-time autoregressive specification with one lag.
    \item \textbf{Univariate Ornstein-Uhlenbeck edge model (OU):} Forecasts are constructed using the same approach as in the MCAR model, but with the state-space $\mathbb{Q}$ restricted to have a diagonal matrix $\mathbf{Q}$, and with a single lag.
\end{enumerate}

Each simulated path contains $3^7 = 2187$ observations, of which the final $400$ are reserved for out-of-sample evaluation. All results are based on $1000$ independent simulations. The total number of observations is chosen to closely match that of the empirical dataset. We evaluate model performance using two metrics. The first is the root mean squared error (RMSE),
\begin{equation*}
    \mathbf{RMSE} = \sqrt{\frac{1}{TK}\sum_{t=1}^T\sum_{k=1}^K\left(\mathbf{Y}^{(k)}_t-\hat{\mathbf{Y}}^{(k)}_t\right)^2}.
\end{equation*}
While this widely used metric captures average squared deviations, it is sensitive to large outliers, which can cause it to behave similarly across models. To address this limitation, we also consider directional accuracy, i.e., the ability to correctly predict the sign of increments, defined as
\begin{equation*}
    \mathbf{DirAcc} = \frac{1}{TK}\sum_{t=1}^T\sum_{k=1}^K\mathbbm{1}_{sgn\left(\mathbf{Y}^{(k)}_t-\mathbf{Y}^{(k)}_{t-1}\right)=sgn\left(\hat{\mathbf{Y}}^{(k)}_t-\mathbf{Y}^{(k)}_{t-1}\right)},
\end{equation*}
which measures the ability to correctly predict the direction of changes. By convention, the naive model achieves a baseline value of $0.5$. This metric provides a more robust, qualitative assessment in the presence of jumps.

Finally, we report computational time, as efficiency is a key practical consideration, particularly in high-frequency settings with large sample sizes and frequent updates. In our framework, the network dimension remains moderate relative to the number of observations.

The orange line in each boxplot represents the median, while the half-violin displays the $5$th to $95$th percentiles of the distribution. As shown in Figures \ref{scale_1_beta_2_rate_1_box}-\ref{scale_10_beta_2_rate_1_box}, the RMSE boxplots do not provide a clear basis for model discrimination, as performance appears broadly similar across models due to the presence of large jump outliers. 

\begin{figure}[htbp]
  \centering
  \includegraphics[scale = .4]{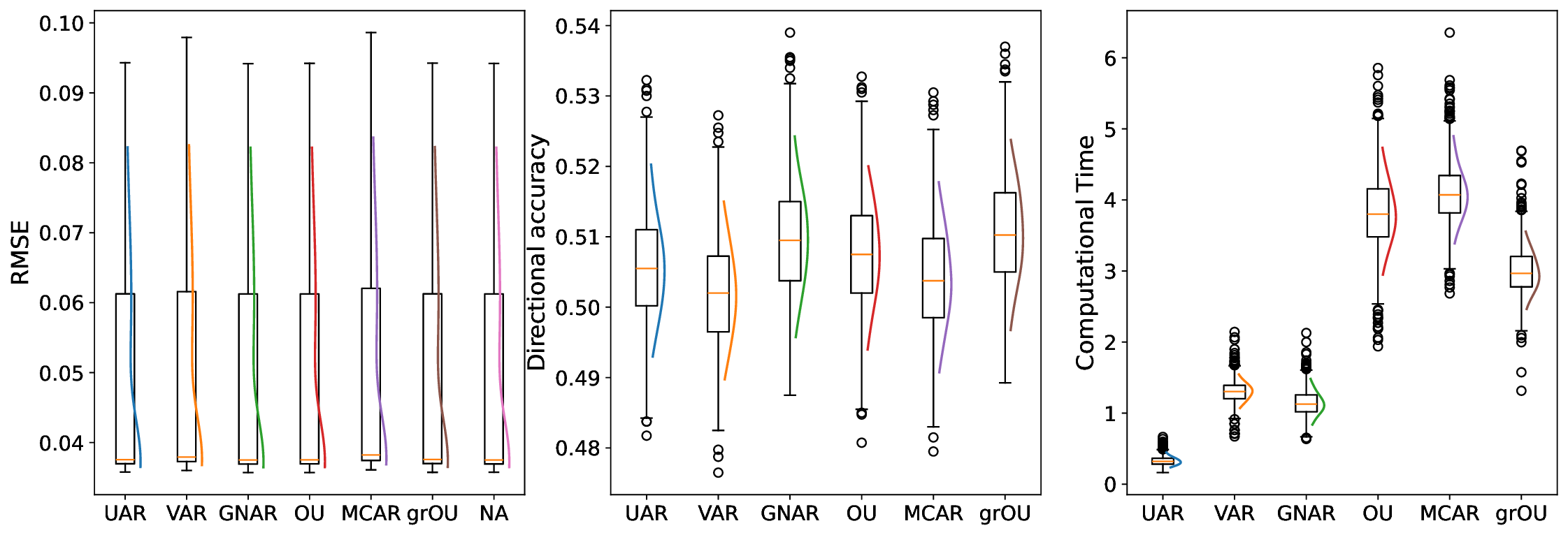}
  \caption{Out-of-sample predictive performance under compound Poisson noise with $\sigma^2 = 1$. Boxplots report RMSE (left), directional accuracy (middle), and computation time (right) across models over 1000 simulated paths.}
  \label{scale_1_beta_2_rate_1_box}
\end{figure}

\begin{figure}[htbp]
  \centering
  \includegraphics[scale = .4]{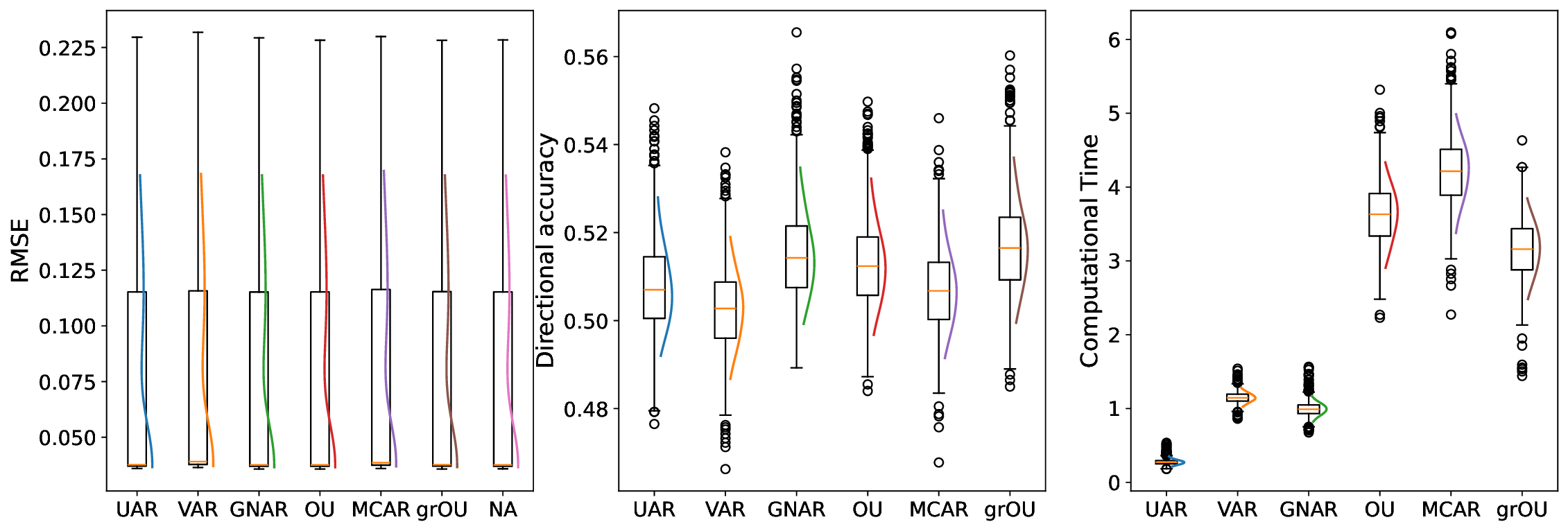}
  \caption{Out-of-sample predictive performance under compound Poisson noise with $\sigma^2 = 5$. Boxplots report RMSE (left), directional accuracy (middle), and computation time (right) across models over 1000 simulated paths.}
  \label{scale_5_beta_2_rate_1_box}
\end{figure}

\begin{figure}[htbp]
  \centering
  \includegraphics[scale = .4]{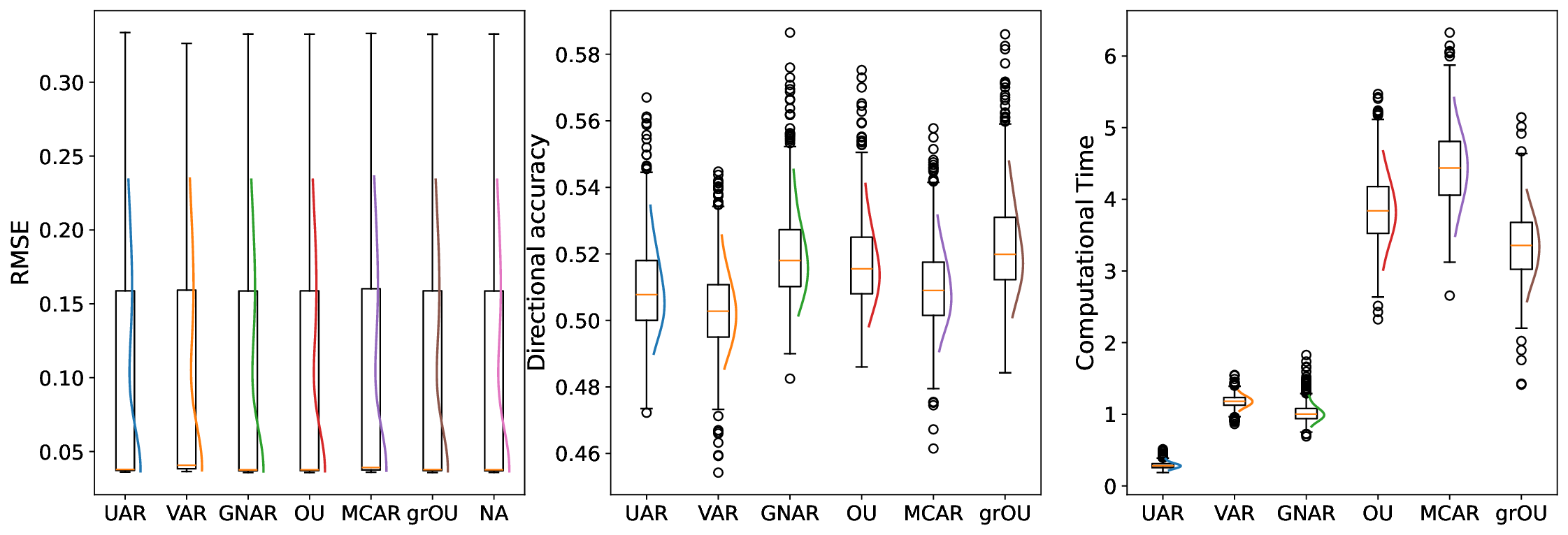}
  \caption{Out-of-sample predictive performance under compound Poisson noise with $\sigma^2 = 10$. Boxplots report RMSE (left), directional accuracy (middle), and computation time (right) across models over 1000 simulated paths.}
  \label{scale_10_beta_2_rate_1_box}
\end{figure}

In contrast, the grOU model outperforms the other benchmark models in terms of directional accuracy, which is consistent with the presence of strong network dependence and jump dynamics. Although GNAR and grOU achieve comparable directional accuracy with lower computational costs, this advantage mainly comes from a refined and well-optimized implementation. With a more refined and optimized algorithm for the continuous-time models, we conjecture that the computational costs could be reduced to the level of the discrete-time versions.

Indeed, a key advantage of the grOU specification lies in its time-discretization scheme: the integral terms are evaluated on a coarse grid, $\Delta_{\mathcal{Q}_t}$, which is coarser than the observation grid, $\Delta_{\mathcal{P}_t}$. This reduces the number of required evaluations and can significantly lower the computational cost, provided the implementation is sufficiently optimized relative to the GNAR model, which operates on the full grid.

Although the VAR model is currently faster than the grOU model in this setting, its computational cost increases quadratically with dimension. Consequently, as the number of edges grows, the VAR model is expected to become less scalable than the grOU model. At the same time, despite its richer multivariate structure, VAR tends to be less competitive in terms of directional accuracy in the presence of potential jump behavior.

By contrast, the simpler AR model requires less computation than the grOU model, but it does not account for cross-edge dependence and therefore ignores potentially important network information. This limitation reduces its directional predictive accuracy. Moreover, as a discrete-time model, it does not accommodate irregularly spaced data as naturally as the grOU model.

Among the continuous-time models implemented within the same code framework, recall that $K$ denotes the number of edges, $L$ the number of lags, and $R_l$ the neighborhood size at lag $l$. The parameter dimensions are $LK$ for the OU model, $LK + \sum_{l=1}^L R_l$ for the grOU model, and $LK^2$ for the MCAR model. 

Although the grOU edge model is more complex than the OU model, it exhibits better computational performance since its flexible network structure leads to faster and more stable convergence under the estimation procedure of \cite[Section $\text{H}.1$]{lucchese2026estimation}. In addition, to avoid singularity issues, we include a ridge regularization term in both the OU and the grOU parameter estimation procedures, see the supplementary code for details. At the same time, the grOU model remains substantially lower-dimensional than the MCAR model, yielding important computational advantages.

In particular, the grOU model attains a higher directional accuracy advantage as jump variability increases. Greater variability in jump size makes jump behavior easier to detect and distinguish from the Brownian component. As a result, the model is more likely to identify the correct direction, which is driven primarily by the continuous, non-jump component.

\subsection{Robustness to model misspecification}
To examine robustness, we introduce network misspecification by omitting certain edges during model estimation.

First, we remove an edge with $\alpha = 1$. Figures \ref{scale_1_beta_2_rate_1_box_miss_1}-\ref{scale_10_beta_2_rate_1_box_miss_1} demonstrate that predictive performance remains largely unaffected. This suggests that the model is robust to moderate misspecification involving weakly persistent edges. If we remove the edge associated with the dominant autoregressive parameter $\alpha = 5$ instead, as shown in Figures \ref{scale_1_beta_2_rate_1_box_miss_5}-\ref{scale_10_beta_2_rate_1_box_miss_5}, this results in only a slight decrease in directional accuracy.

\begin{figure}[htbp]
  \centering
  \includegraphics[scale = .4]{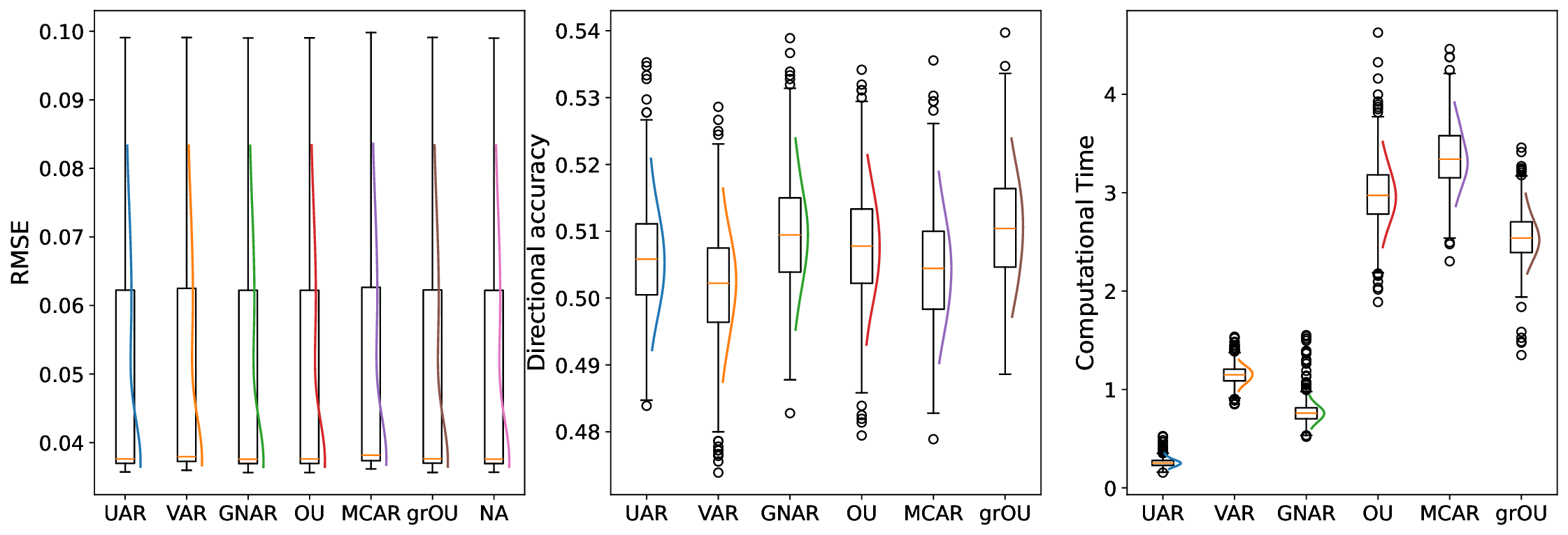}
  \caption{Predictive performance when a weakly persistent edge ($\alpha = 1$) is omitted during estimation ($\sigma^2 = 1$). Boxplots report RMSE (left), directional accuracy (middle), and computation time (right) over 1000 simulations.}
\label{scale_1_beta_2_rate_1_box_miss_1}
\end{figure}

\begin{figure}[htbp]
  \centering
  \includegraphics[scale = .4]{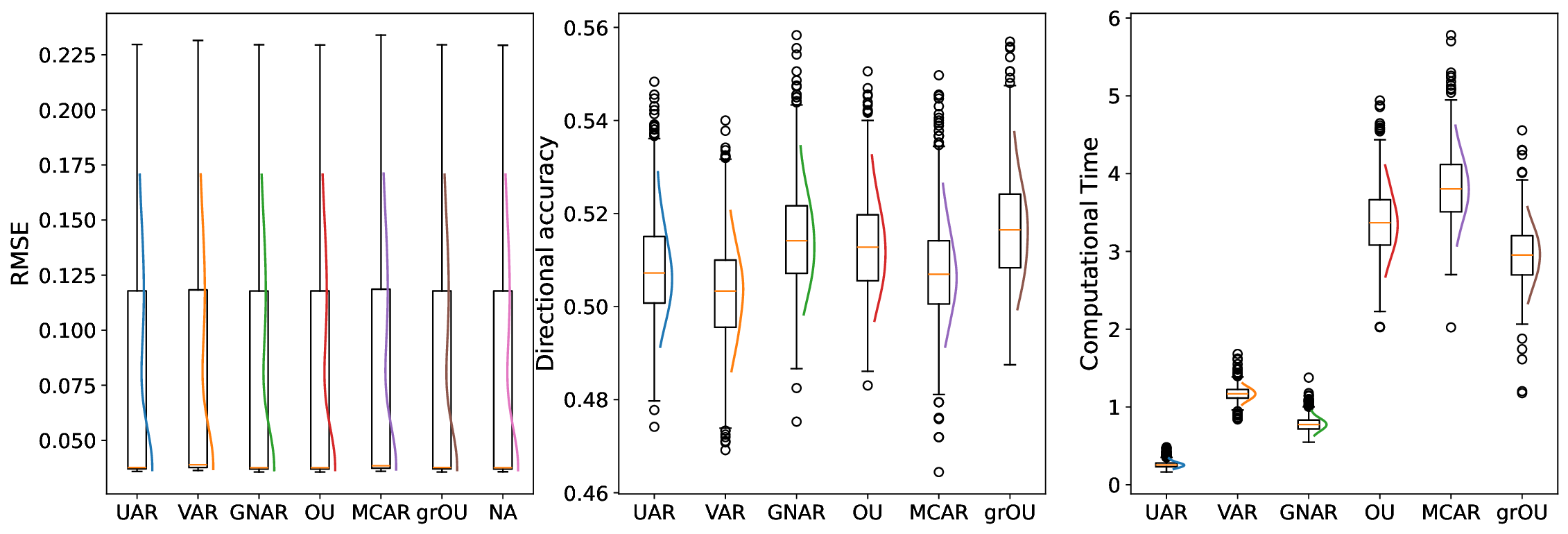}
  \caption{Predictive performance when a weakly persistent edge ($\alpha = 1$) is omitted during estimation ($\sigma^2 = 5$). Boxplots report RMSE (left), directional accuracy (middle), and computation time (right) over 1000 simulations.}
\label{scale_5_beta_2_rate_1_box_miss_1}
\end{figure}

\begin{figure}[htbp]
  \centering
  \includegraphics[scale = .4]{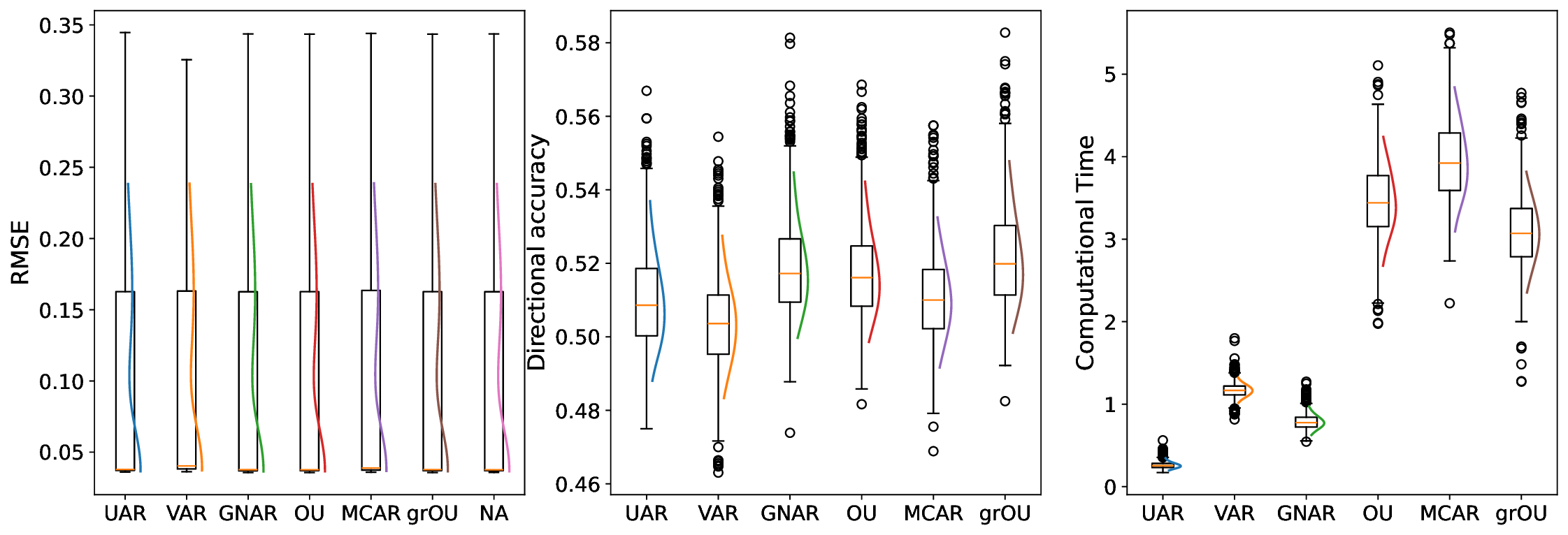}
  \caption{Predictive performance when a weakly persistent edge ($\alpha = 1$) is omitted during estimation ($\sigma^2 = 10$). Boxplots report RMSE (left), directional accuracy (middle), and computation time (right) over 1000 simulations.}
\label{scale_10_beta_2_rate_1_box_miss_1}
\end{figure}

\begin{figure}[htbp]
  \centering
  \includegraphics[scale = .4]{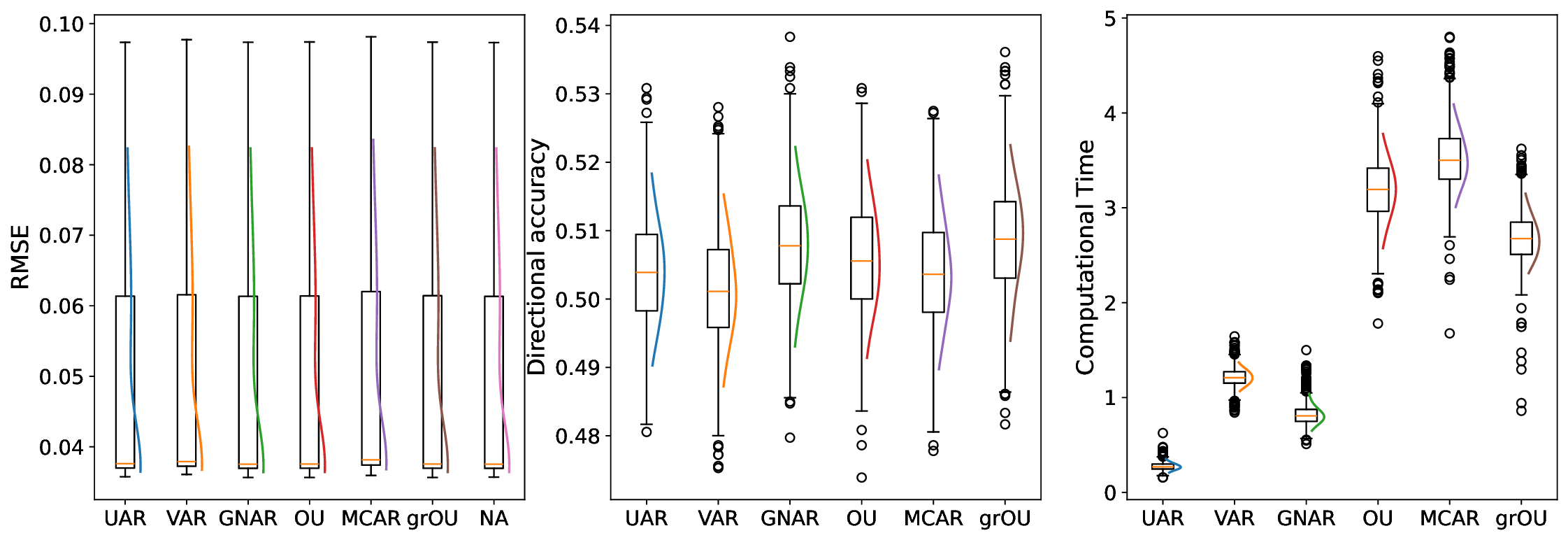}
  \caption{Predictive performance when the dominant edge ($\alpha = 5$) is omitted during estimation ($\sigma^2 = 1$). Boxplots report RMSE (left), directional accuracy (middle), and computation time (right) over 1000 simulations.}
  \label{scale_1_beta_2_rate_1_box_miss_5}
\end{figure}

\begin{figure}[htbp]
  \centering
  \includegraphics[scale = .4]{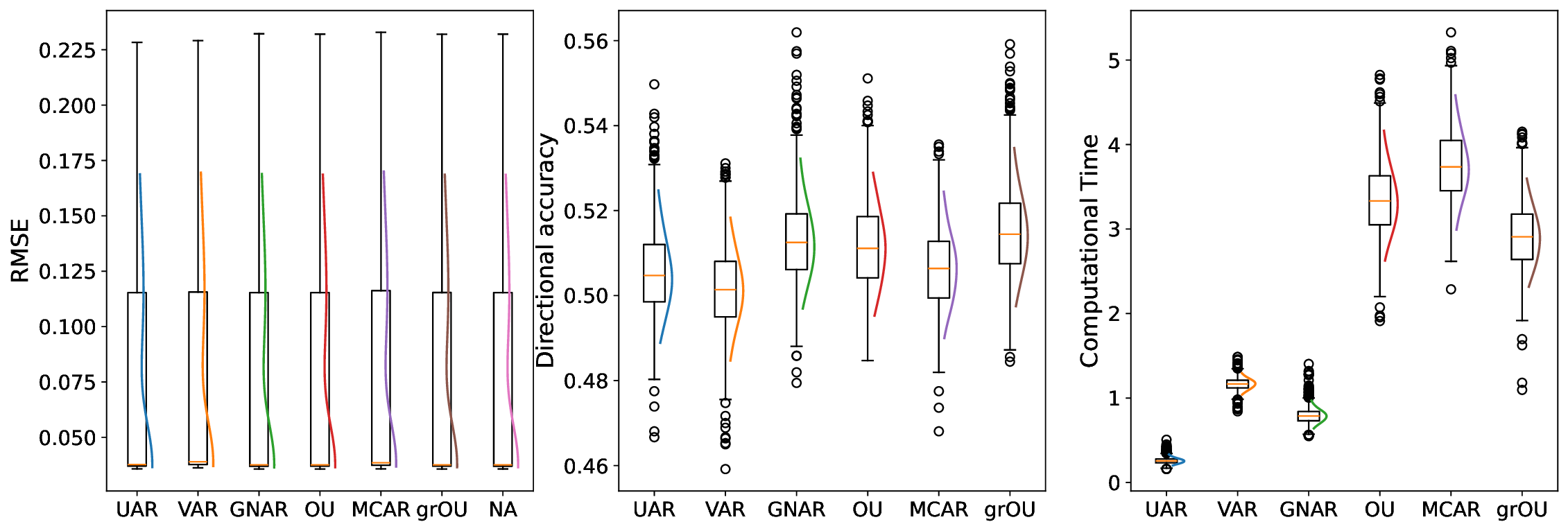}
  \caption{Predictive performance when the dominant edge ($\alpha = 5$) is omitted during estimation ($\sigma^2 = 5$). Boxplots report RMSE (left), directional accuracy (middle), and computation time (right) over 1000 simulations.}
  \label{scale_5_beta_2_rate_1_box_miss_5}
\end{figure}

\begin{figure}[htbp]
  \centering
  \includegraphics[scale = .4]{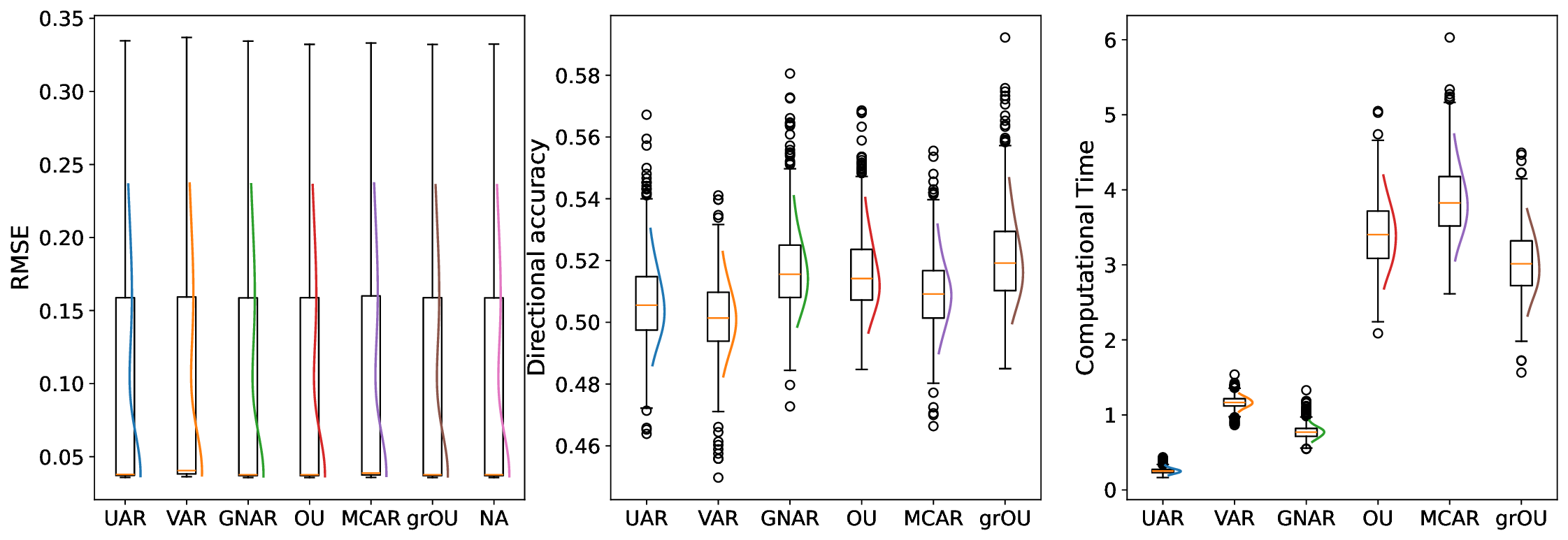}
  \caption{Predictive performance when the dominant edge ($\alpha = 5$) is omitted during estimation ($\sigma^2 = 10$). Boxplots report RMSE (left), directional accuracy (middle), and computation time (right) over 1000 simulations.}
  \label{scale_10_beta_2_rate_1_box_miss_5}
\end{figure}

Overall, the results for the missing-edge scenarios demonstrate that the grOU model is robust to moderate network misspecification, and this pattern holds uniformly across $\sigma^2 = 1, 5, 10$. Whether the omitted edge corresponds to the stronger connection $(\alpha = 5)$ or the weaker connection $(\alpha = 1)$, the changes in RMSE and directional accuracy are relatively small. This suggests that the predictive gains from incorporating network information are not overly sensitive to small structural errors in the network. These findings provide evidence that grOU is not only accurate under correct specification but also robust when the network is partially misspecified, as shown in Figures \ref{boxplot_grOU_1}-\ref{boxplot_grOU_10}. Appendix \ref{Simulation data table} provides summaries of predictive accuracy and computational cost, confirming that the proposed model achieves a favorable balance between statistical performance and efficiency.

\begin{figure}[htbp]
  \centering
  \includegraphics[scale = .5]{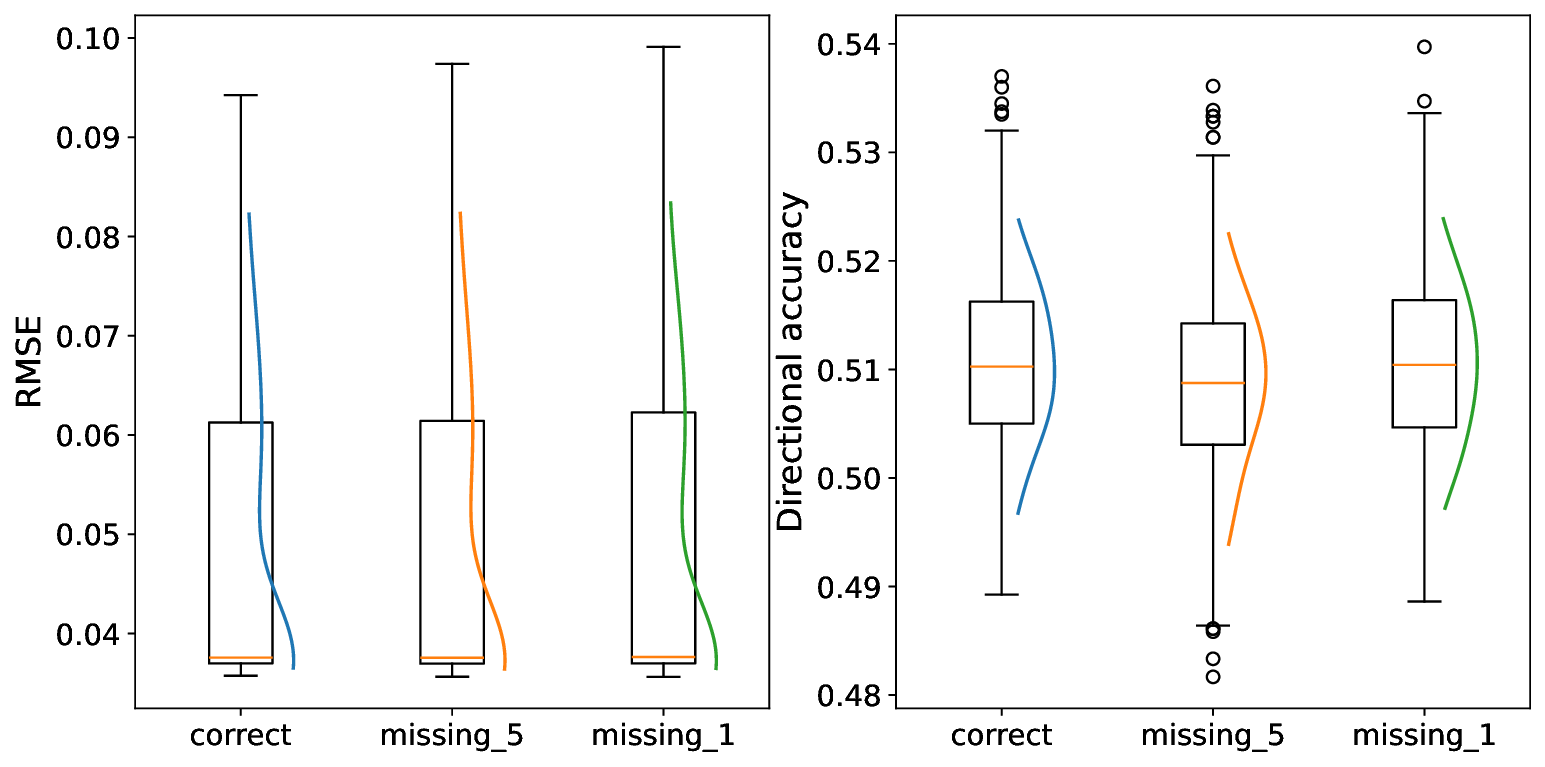}
  \caption{Predictive performance across all three scenarios ($\sigma^2 = 1$). Boxplots report RMSE (left) and directional accuracy (right) over 1000 simulations.}
  \label{boxplot_grOU_1}
\end{figure}

\begin{figure}[htbp]
  \centering
  \includegraphics[scale = .5]{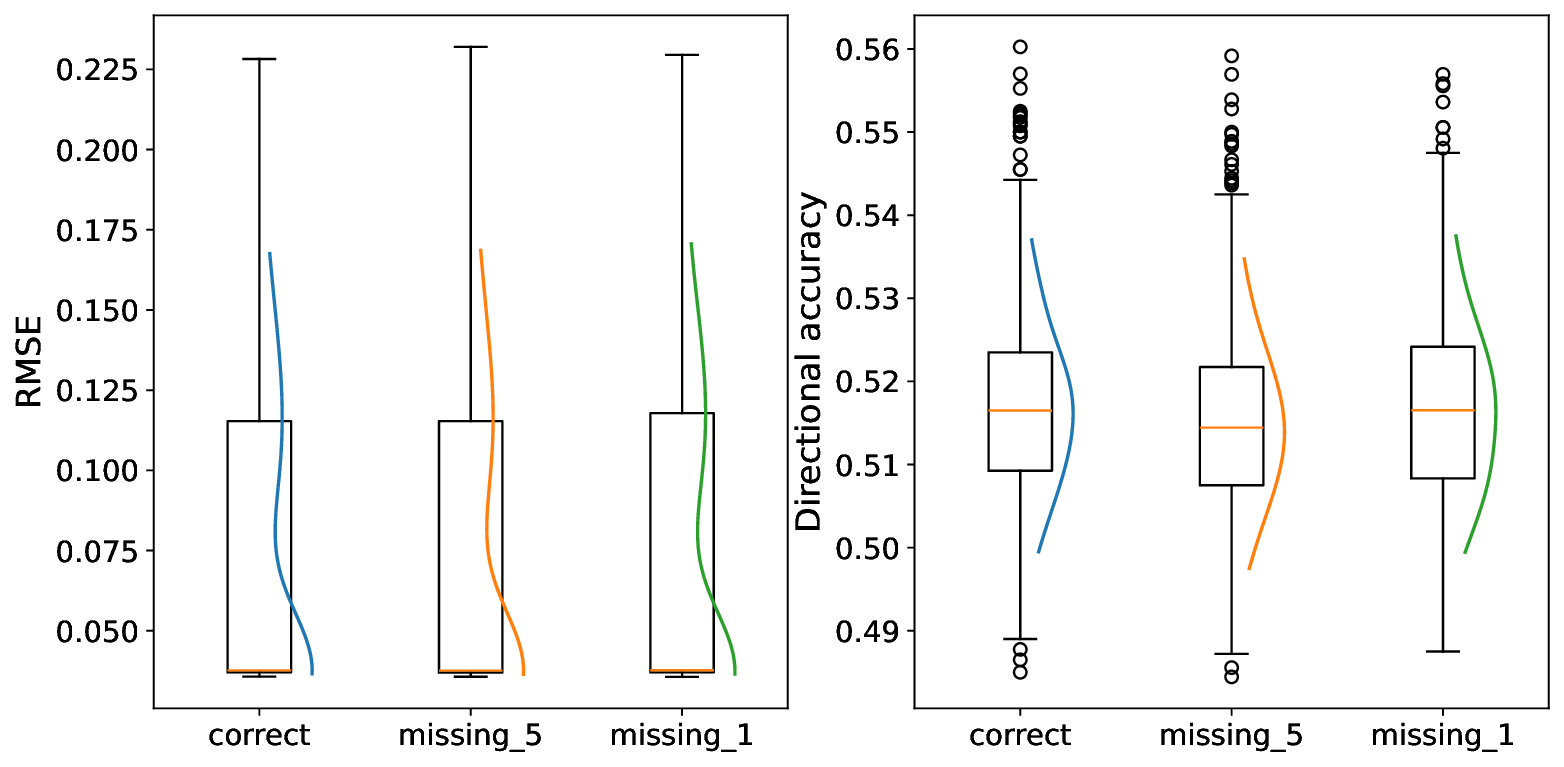}
  \caption{Predictive performance across all three scenarios ($\sigma^2 = 5$). Boxplots report RMSE (left) and directional accuracy (right) over 1000 simulations.}
  \label{boxplot_grOU_5}
\end{figure}

\begin{figure}[htbp]
  \centering
  \includegraphics[scale = .5]{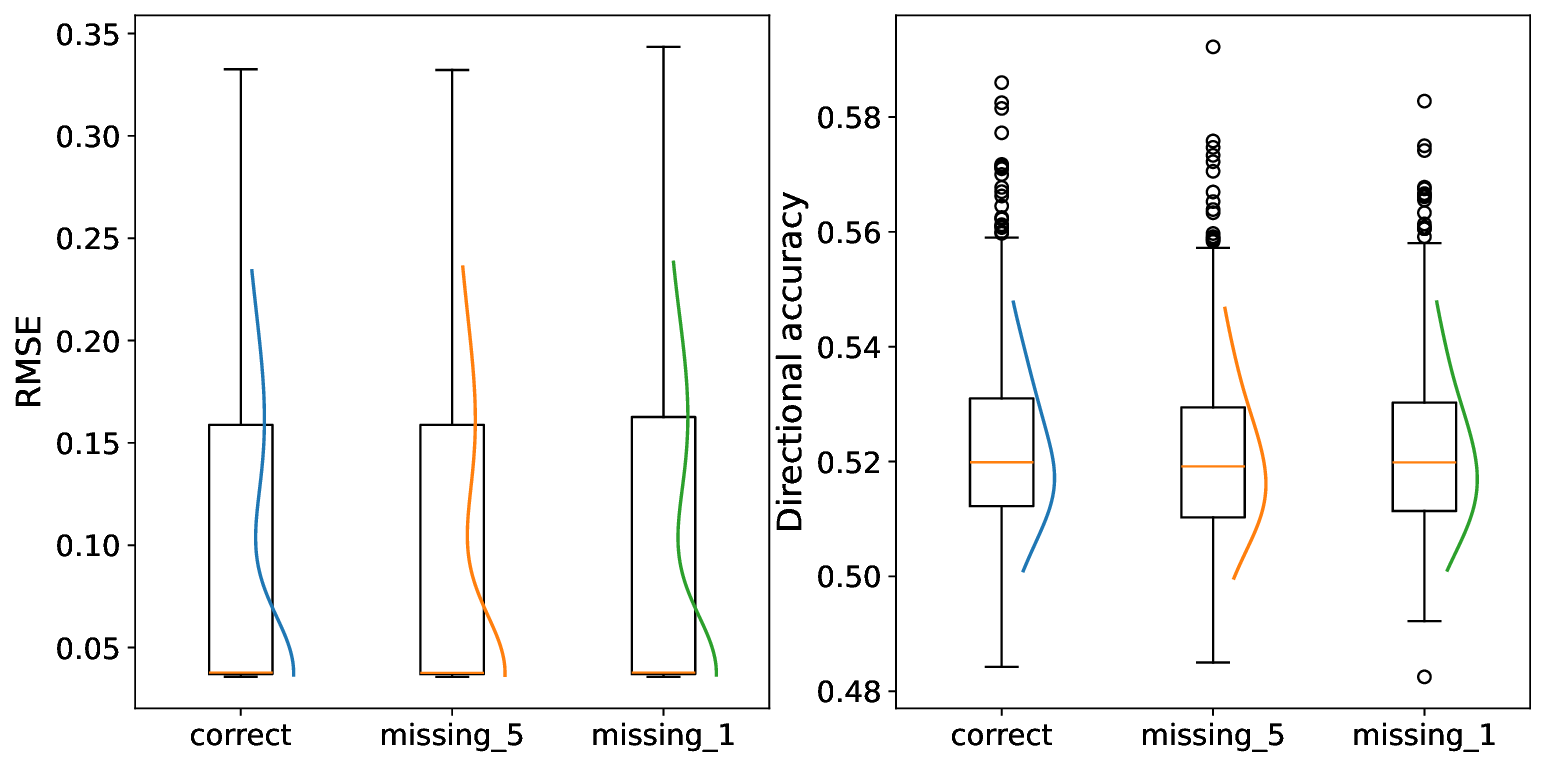}
  \caption{Predictive performance across all three scenarios ($\sigma^2 = 10$). Boxplots report RMSE (left) and directional accuracy (right) over 1000 simulations.}
  \label{boxplot_grOU_10}
\end{figure}

\section{Empirical illustration}
\label{Empirical illustration}
To illustrate the empirical properties of the grOU edge process developed above, we retrieve high-frequency limit order book data for a variety of assets from the LOBSTER database. It is well documented that more frequently traded equities typically exhibit lower levels of microstructure noise. Accordingly, we select equities from the S\&P 500 and NASDAQ constituents as of January 1, 2023. In addition, we include SPDR S\&P 500 ETF Trust (ticker: SPY) and NDAQ (Nasdaq, Inc.). The remaining equities are drawn from multiple sectors to ensure economic heterogeneity while preserving potential structural connections. Specifically, we include:
\begin{itemize}
    \item Financial: JPM (JPMorgan Chase \& Co.), GS (Goldman Sachs Group, Inc.)
\item Consumer Staples: KO (The Coca-Cola Company), MNST (Monster Beverage Corporation)
\item Technology: MSFT (Microsoft Corporation), AMZN (Amazon.com, Inc.)
\end{itemize}
This selection yields an undirected graph with eight vertices. Economic and index-based relationships among these assets naturally induce potential edges, forming a moderately sized network suitable for empirical analysis. Importantly, we do not impose a fixed static network structure. 

Following the methodology from \cite{lucchese2024short}, we preprocess the LOBSTER data into synchronized logarithmic mid-quote series sampled at one-second frequency using the refresh-time sampling scheme to address asynchronicity across assets. Within each interval, we record the last observed mid-quote price. Several cleaning procedures are implemented to mitigate microstructure noise and non-stationarity: Elimination of crossed quotes; compression of multi-record conceptual events (e.g., limit order modifications represented as cancellation followed by immediate submission; market orders matched against multiple resting orders) into single timestamps; exclusion of observations outside regular market hours; removal of the first and last 60 minutes of trading to reduce opening and closing effects.

Standard realized covariance estimators are known to suffer from the Epps effect (\cite{epps1979comovements}) at high sampling frequencies. To mitigate asynchronicity, we first synchronize the logarithmic mid-quote data using refresh time sampling. We then apply a pre-averaging procedure and compute hourly covariance via the Modulated Realized Covariance (MRC) estimator introduced by \cite{christensen2010pre}. This approach is designed to reduce the impact of market microstructure noise while preserving high-frequency information. A detailed justification for the choice of the MRC estimator, along with a description of the implementation, is provided in Appendix \ref{mrc}.

We use both message and quote data over the sample period from 2023-01-01 to 2026-02-01. The data are divided into a training set (2023-01-01 to 2025-06-30) and a testing set (2025-07-01 to 2026-02-01). Figure \ref{corr_cov} reports hourly pre-averaged correlation and covariance between SPY and NDAQ from 2023-01-01 to 2025-07-01. The correlation exhibits few abrupt shifts, whereas the covariance displays several jumps over time, including a cluster of large moves around April 2025. Motivated by this behavior, Section \ref{Simulation study} evaluates predictive performance and robustness under different jump regimes to better reflect empirical variability.

\begin{figure}[htbp]
  \centering
  \includegraphics[scale = .43]{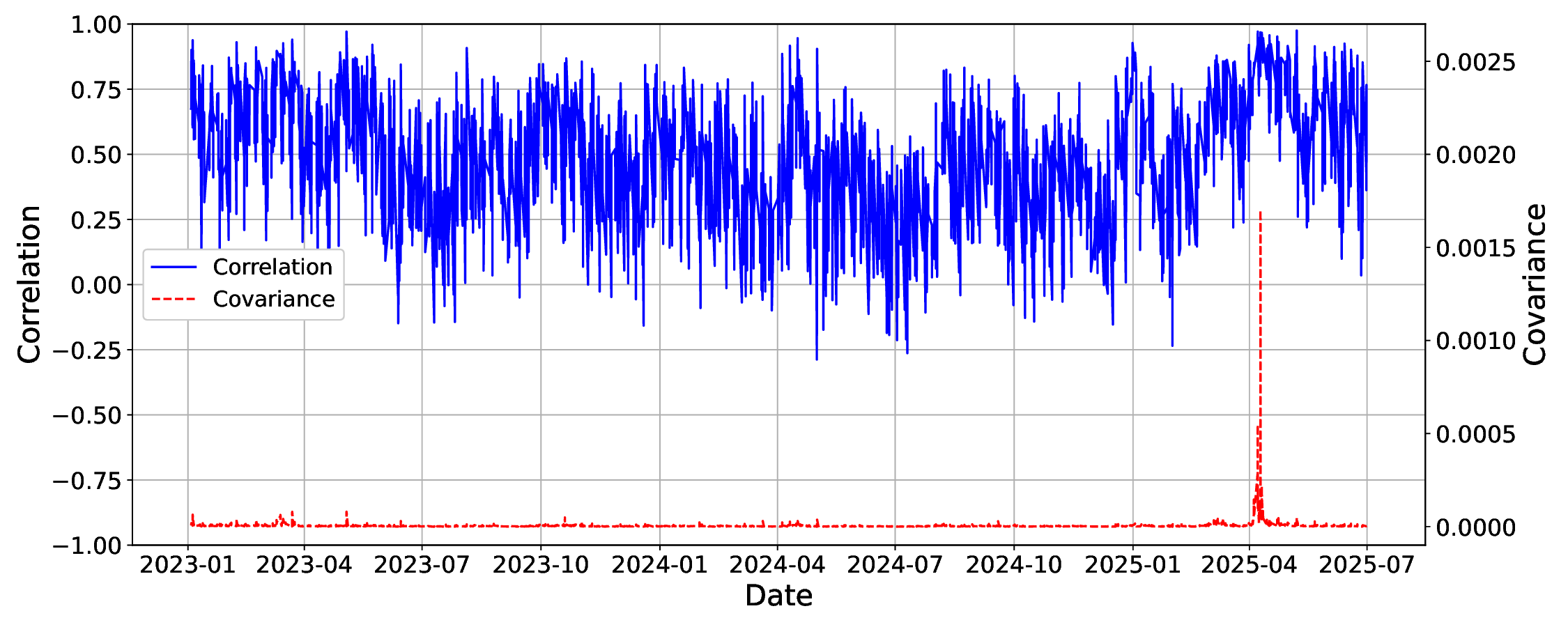}
  \caption{Hourly correlation (blue) and covariance (red) between SPY and NDAQ based on one-second mid-quote sampling from 2023-01-01 to 2025-07-01.}
  \label{corr_cov}
\end{figure}

\subsection{Initial value determination}
In the discrete-time framework with $p$ lags, prediction is based on the most recent $p$ observations, which serve as the required initial conditions. In contrast, the continuous-time counterpart does not directly provide the last $p-1$ time derivatives; only the most recent observation, corresponding to the $0$- th derivative, is observed.  

To bridge this gap, we initialize all unobserved state components at zero and retain only the observed component corresponding to the $0$-th derivative. This provides a simple and stable initialization strategy, which imposes minimal prior assumptions when information about higher-order components is unavailable, thereby reducing the risk of misleading predictions. By doing so, the model produces a locally flat forecast at the outset, but improves comparability across model specifications and provides a baseline for empirical implementation.

The effective sample size $N$ is determined by the grid ratio between $\mathcal{P}_t$ and $\mathcal{Q}_t$. As shown in Theorem \ref{thm2}, time derivatives are computed via finite differences on a fine grid $\mathcal{P}_t$, while stochastic integrals are evaluated on a coarser grid $\mathcal{Q}_t$. Finite differences amplify high-frequency noise yet incur limited bias, motivating a fine grid $\mathcal{P}_t$, whereas It\^o integrals on overly fine meshes suffer variance inflation (noise-dominated increments, unstable jump filtering), so coarser $\mathcal{Q}_t$ yields more stable estimates. Consequently, consistency requires the derivative error to vanish faster, which leads to a nesting condition in which $\mathcal{P}_t$ refines faster than $\mathcal{Q}_t$ as imposed in Appendix \ref{discrete}. In practice, the estimator may drift as $\mathcal{Q}_t$ varies due to microstructure noise and model misspecification. We therefore adopt a relatively large $\Delta_{\mathcal{P}_t}=0.01$ and a grid ratio between $\Delta_{\mathcal{P}_t}$ and $\Delta_{\mathcal{Q}_t}$ of $18$, a pragmatic compromise that improves numerical stability and robustness at a modest cost to theoretical optimality.

We structure the empirical illustration into two cases: one incorporating a network and one without. We begin with the network-based specification.
\subsection{Model with network structure}
\label{With a network}
Suppose we are interested in predicting covariance along specific edges among a set of assets connected through a given network. Edge data are extracted from the network to construct a vector time series. For model selection, we primarily choose the model with the strongest predictive performance in terms of directional accuracy. For high-frequency data, extending the neighborhood order beyond immediate neighbors increases the risk of overfitting due to strong edge interactions. Accordingly, we recommend $R_l=1$ for each lag $l$ as a default. If tuning is necessary, we restrict the search to $R_l \leq 2$. Among the grOU specifications, we select the model over all lag orders $L=1,2,3$, where the maximum neighborhood stage is uniformly set to either $1$ or $2$ across all lags.

If two models exhibit similar out-of-sample performance within a tolerance of $10^{-2}$, we can then use the Bayesian Information Criterion (BIC), as in the GNAR node-based framework from \cite{knight2020generalized}:
\begin{equation}
\mathbf{BIC} = -\hat{\theta}^\top\left[\mathbf{K}\right]\hat{\theta}+\left(LK+\sum_{l=1}^LR_l\right)\log(N), 
\end{equation}
where $\hat{\theta}^\top\!\left[\mathbf{K}\right]\hat{\theta}$ is the log-likelihood from \eqref{likelihoodquad}, $LK+\sum_{l=1}^L R_l$ is the total number of drift parameters in the grOU edge model, and $N$ is the number of coarse-grid increments used in the stochastic integral approximation (not the raw number of timestamps).

Beyond BIC-based selection, \cite{10.1093/jrsssa/qnaf182} proposed a lag selection method based on corbit plots for discrete-time models; however, it remains an open problem whether an analogous approach can be extended to the continuous-time setting.

As an illustration, we consider the network in Figure \ref{network}.
\begin{figure}[htbp]
  \centering
  \includegraphics[scale =.4]{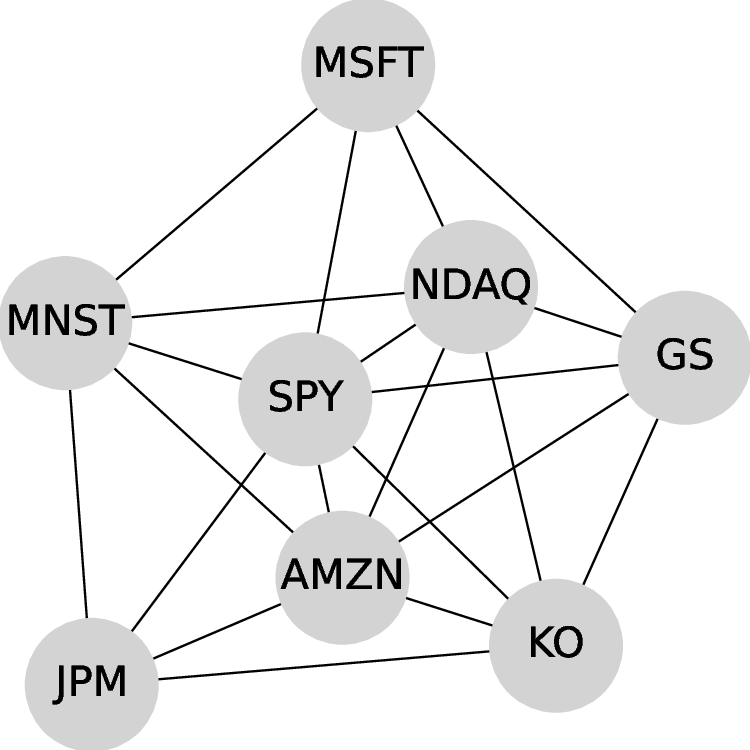}
  \caption{A sample network on eight assets.}
  \label{network}
\end{figure}
We compute the one-step forecasting directional accuracy (as in Section \ref{Simulation study}) for all competing models defined in Section \ref{Predictive performance}, as well as for additional grOU specifications with higher lag orders and larger neighborhood stages. The results, presented in Table \ref{table_pred1}, suggest that the grOU(1,[1]) model is the strongest choice when directional prediction is the main priority. They also indicate that increasing the order of the grOU model does not improve predictive performance in this setting.

Although predictive performance is our primary criterion, the difference in predictive performance between grOU(1,[1]) and grOU(1,[2]) falls within the tolerance threshold of $10^{-2}$. The BIC results shown in Figure \ref{bic} indicate that grOU(1,[2]) is the optimal specification for this network.

\begin{table}[htbp]
\centering
\begin{tabular}{|c|cc|}
\hline
Model & RMSE($10^{-6}$) & DirAcc\\ 
\hline
NA & 7.2826 & 0.5\\
AR & 6.4733 & 0.5871\\
VAR & 13.3647 & 0.5736\\
GNAR & 6.6976 & 0.6487\\
OU & 7.2133 & 0.6617\\
MCAR & 7.1677 & 0.6078\\
grOU(1,[1]) & 7.2175 & 0.6717\\
grOU(1,[2]) & 7.2180 & 0.6687\\
grOU(2,[1,1])& 7.2837 & 0.5014\\
grOU(2,[2,2])& 7.2840 & 0.4986\\
grOU(3,[1,1,1])& 7.2804 & 0.5186\\
grOU(3,[2,2,2])& 7.2804 & 0.5234\\
\hline
\end{tabular}
\caption{The one-step forecast results for competing models under the network in Figure \ref{network}.}
\label{table_pred1}
\end{table}

\begin{figure}[htbp]
  \centering
  \includegraphics[scale =.5]{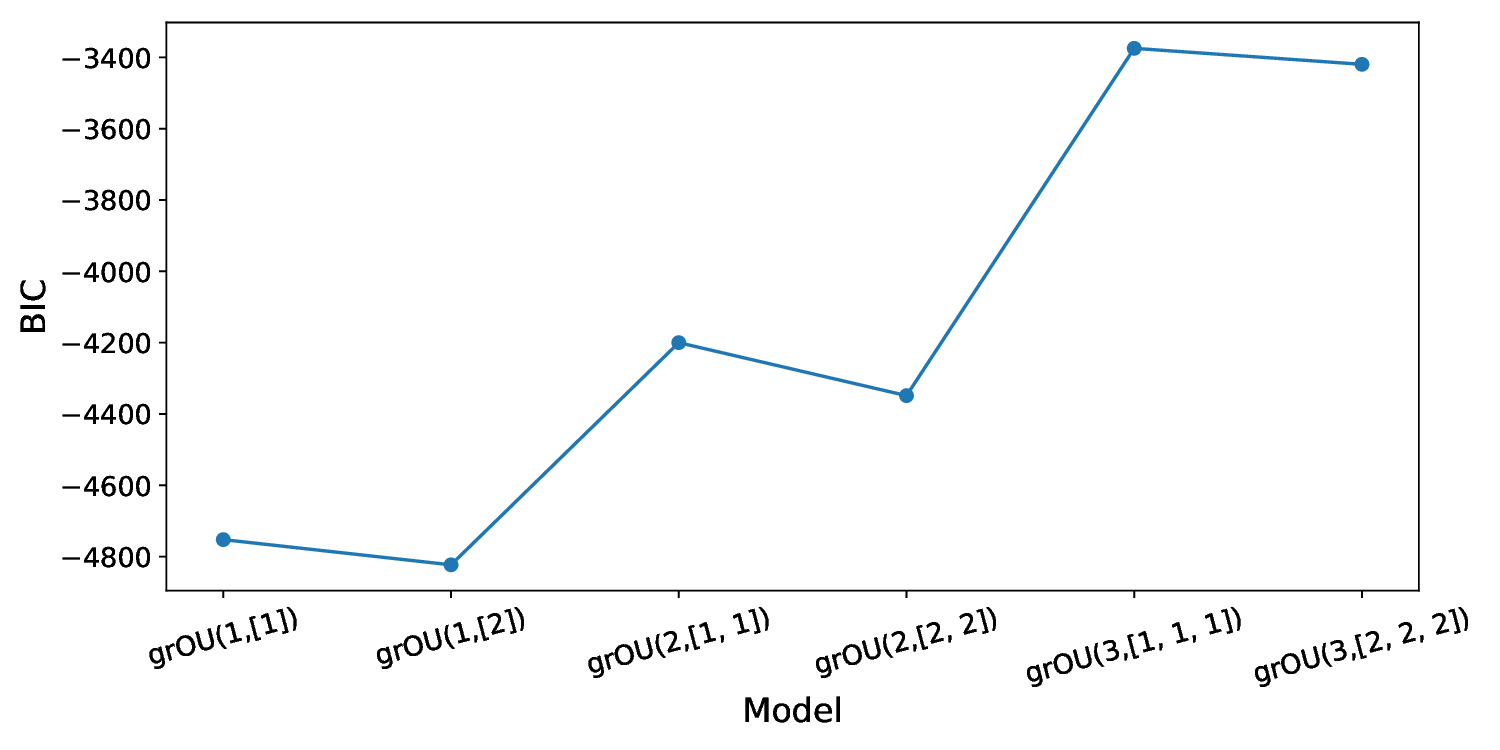}
  \caption{Bayesian Information Criterion values for the set of candidate edge-based grOU models. The plot summarizes model selection results across increasing model complexity for structures based on repeated $1$s and repeated $2$s.}
  \label{bic}
\end{figure}

\subsection{Model without network structure}
\label{Without a network}
Now suppose no network is given, and we wish to jointly select a network and a model to maximize predictive performance. We start with a simple grOU edge specification and select the best-fitting network from a randomly generated candidate set. For example, we initialize with grOU(1,[1]), and, similar to \cite{nason2025forecasting}, we generate 1000 Erd\H{o}s–R\'enyi graphs with an edge probability of $0.4$, which induces a high likelihood of connectivity in the graph. For each candidate network, we compute the directional accuracy for the screening model (grOU(1,[1])). However, our simulation study also indicates that the grOU model remains stable under mild network misspecification, even in the presence of strong autoregressive dependence. Accordingly, this procedure is better viewed as a method for narrowing the set of plausible networks than as a precise diagnostic for recovering the true network.

Our initial idea, later found to have a limitation, was that once a network had been selected, we would proceed as in Subsection \ref{With a network} to determine the optimal model for that network. If the optimal model coincided with the initial choice, we would terminate the selection procedure. Otherwise, we would update the model to the newly selected specification and iterate the process as long as out-of-sample predictive performance continued to improve.

A potential limitation of the iterative graph-then-model selection scheme is that it may converge to a locally optimal graph-model pair rather than the globally best combination. In particular, if a graph $G_1$ performs best for a given model $M_1$, and $M_1$ in turn performs best among candidate models conditional on $G_1$, this does not guarantee that the pair $(G_1,M_1)$ outperforms another pair $(G_2,M_2)$. 

To address this issue more effectively, we jointly select the graph and the model in the final stage. Specifically, we first employ a simple grOU specification, such as grOU$(1,[1])$, to screen a large collection of candidate networks and retain the most promising ones according to out-of-sample predictive performance. We then carry out model selection, as described in Section \ref{With a network}, separately for each retained network. The final specification is thus the pair $(\hat{G},\hat{M})$ that achieves the best predictive performance over the reduced joint search space, rather than one obtained solely through alternating conditional optimization. The overall procedure is summarized as follows:
\begin{enumerate}
    \item Initialize a simple grOU specification, for example, grOU(1,[1]).
    \item Generate a candidate set of networks, for example, 1000 Erd\H{o}s-R\'enyi graphs with an edge probability of 0.4.
    \item Evaluate the initial model on each candidate network and retain the top 50 candidates.
    \item For each retained network, select the best model specification within the grOU edge family using the model selection criterion described in Section \ref{With a network}.
    \item Choose the pair $(\hat G,\hat M)$ that delivers the best predictive performance among all candidate pairs.
\end{enumerate}

The selected network in Figure \ref{selected_network} is relatively sparse but still connects the main cross-asset relationships, suggesting that the forecasting gains in Table \ref{table_pred2} come from retaining the most informative dependencies rather than imposing a dense graph. Under this selected network, the grOU model delivers the strongest directional accuracy while remaining competitive in RMSE. In particular, grOU(1,[2]) slightly outperforms OU and GNAR in directional prediction, as it combines a network structure missing from OU with a continuous-time framework that is better able to capture jump dynamics than GNAR.

\begin{figure}[htbp]
  \centering
  \includegraphics[scale =.4]{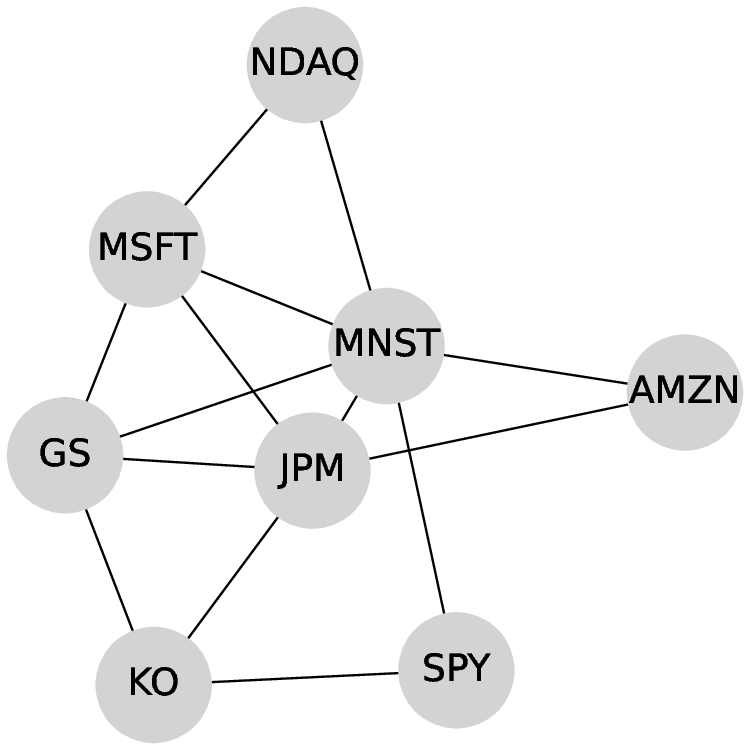}
  \caption{The selected network on eight assets.}
  \label{selected_network}
\end{figure}

\begin{table}[htbp]
\centering
\begin{tabular}{|c|cc|}
\hline
Model & RMSE($10^{-6}$) & DirAcc\\ 
\hline
NA & 9.0310 & 0.5\\
AR & 7.6085 & 0.5997\\
VAR & 12.0966 & 0.5772\\
GNAR & 7.9201 & 0.6878\\
OU & 8.9642 & 0.6923\\
MCAR & 8.9341 & 0.6121\\
grOU(1,[2]) & 8.9568 & 0.6982\\

\hline
\end{tabular}
\caption{The one-step forecast results for competing models under the network in Figure \ref{selected_network}.}
\label{table_pred2}
\end{table}

\section{Conclusion and outlook}
\label{Conclusion and Outlook}
We introduced a L\'evy-driven graph Ornstein-Uhlenbeck (grOU) edge framework for modeling high-dimensional network time series in continuous time. The model inherits the flexibility of Ornstein-Uhlenbeck dynamics while encoding cross-edge dependence through weighted neighborhood operators, thus extending GNAR-edge and MCAR formulations to a tractable continuous-time setting. We reformulated likelihood-based estimators under both continuous and discretized observation schemes and showed their consistency and asymptotic normality under standard conditions. 

Our simulation studies demonstrate the progressive convergence of the discretized maximum likelihood estimator for the grOU edge model as the time horizon increases. Moreover, the new modeling framework achieves higher directional accuracy in the presence of jumps and network dependence, while reducing computational costs relative to multivariate continuous-time alternatives. We also note that the framework can accommodate mild forms of network misspecification with only a modest decrease in predictive accuracy.

An empirical study shows that the grOU model achieves superior directional accuracy at a realistic computational cost for a given network, while BIC-based selection favors parsimonious lag and neighborhood structures. When the underlying network is unobserved, a simple search over candidate graphs, guided by directional accuracy, offers a practical approach to joint network and model selection.

Overall, the grOU edge framework offers a parsimonious, stable and computationally efficient approach for forecasting jump-driven, high-frequency network processes. In future research, it will be interesting to consider data-driven weighting of edge interactions and inference for time-varying network topologies. Moreover, it would be natural to include a stochastic volatility term in the model to allow for additional flexibility and an even greater choice in marginal distributions.

\appendix
\section{Proof of Proposition \ref{prop:secondorder-cond}}
\label{prop_proof}
\begin{proof}
We use the independent increments property of $\textbf{L}_t$ in the following and find that 
\begin{align*}
    \mathbb{E}[\textbf{Y}_{t+h}\vert\mathcal{F}_t] &= \mathbb{E}\left[\textbf{A}(\mathbf{X}_{t+h}-e^{h\mathbb{Q}}\mathbf{X}_t)\vert\mathcal{F}_t\right] + \textbf{A}e^{h\mathbb{Q}}\mathbf{X}_{t}
    =\int_{t}^{t+h}\mathbf{A}e^{(t+h-s)\mathbb{Q}}\mathcal{E}\boldsymbol{\mu}_{\mathbf{L}}ds+\textbf{A}e^{h\mathbb{Q}}\mathbf{X}_{t}\\
    &=
\mathbf{A}\mathbb{Q}^{-1}\left[e^{h\mathbb{Q}}-\textbf{I}\right]\mathcal{E}\boldsymbol{\mu}_{\mathbf{L}} + \textbf{A}e^{h\mathbb{Q}}\mathbf{X}_{t}
=\mathbf{A}[\mathbb{Q}^{-1} e^{h\mathbb{Q}}\mathcal{E}\boldsymbol{\mu}_{\mathbf{L}} +e^{h\mathbb{Q}}\mathbf{X}_{t}]  
-\mathbf{A}\mathbb{Q}^{-1}\mathcal{E}\boldsymbol{\mu}_{\mathbf{L}}\\
&= \E[\mathbf{Y}_0]+\mathbf{A}[\mathbb{Q}^{-1} e^{h\mathbb{Q}}\mathcal{E}\boldsymbol{\mu}_{\mathbf{L}} +e^{h\mathbb{Q}}\mathbf{X}_{t}]=\E[\mathbf{Y}_0]+\mathbf{A}e^{h\mathbb{Q}}[\mathbb{Q}^{-1} \mathcal{E}\boldsymbol{\mu}_{\mathbf{L}} +\mathbf{X}_{t}]\\
&=\E[\mathbf{Y}_0]
+\mathbf{A}e^{h\mathbb{Q}}[\mathbf{X}_{t}-\E[\mathbf{X}_0]], \quad \mathrm{where}\,\, \E[\mathbf{X}_0]=-\mathbb{Q}^{-1} \mathcal{E}\boldsymbol{\mu}_{\mathbf{L}}, \\
    \text{Var}[\textbf{Y}_{t+h}\vert\mathcal{F}_t] &= \text{Var}\left[\textbf{A}(\mathbf{X}_{t+h}-e^{h\mathbb{Q}}\mathbf{X}_t)\vert\mathcal{F}_t\right]=\int_{t}^{t+h}\mathbf{A}e^{(t+h-s)\mathbb{Q}}\mathcal{E}\boldsymbol{\Sigma}_L\mathcal{E}^{\top} e^{(t+h-s)\mathbb{Q}^{\top}}\mathbf{A}^{\top} ds\\
    &=\mathbf{A} \left(\int_{t}^{t+h}e^{(t+h-s)\mathbb{Q}}\mathcal{E}\boldsymbol{\Sigma}_L\mathcal{E}^{\top} e^{(t+h-s)\mathbb{Q}^{\top}}ds\right)\mathbf{A}^{\top}\\    
    &=\mathbf{A} \left(\int_{0}^{h}e^{u\mathbb{Q}}\mathcal{E}\boldsymbol{\Sigma}_L\mathcal{E}^{\top} e^{u\mathbb{Q}^{\top}}du\right)\mathbf{A}^{\top}.
\end{align*}
\end{proof}

\section{Proof of Theorem \ref{thm}}
\label{thm_proof}
\begin{proof}
Under Definition \ref{groudef}, the grOU edge process admits the state-space representation
\begin{equation*}
d\mathbf{X}_t = \mathbb{Q}(\theta)\mathbf{X}_tdt + \mathcal{E}d\mathbf{L}_t,~\mathbf{Y}_t = \mathbf{AX}_t,
\end{equation*}
with $\mathbb{Q}(\theta)\in M_K^-$ and parameter restricted to the graph-structured set $\Theta$. Hence the model is a special case of the multivariate continuous-time autoregressive framework studied in \cite{lucchese2026estimation}. Using the L\'evy–It\^o decomposition of the driving L\'evy process $\{\mathbf{L}_s,s\geq 0\}$, the state equation \eqref{model2} admits the representation
\begin{equation}
\label{model3}
    d\mathbf{X}_t = (\mathbf{\mathbb{Q}}\mathbf{X}_t+\mathcal{E}\mathbf{b})dt + \mathcal{E}\Sigma^{1/2}d\mathbf{W}_t + \int_{\Vert z\Vert\leq 1}\mathcal{E}z\{\mu(dt,dz)-\nu(dz)dt\} + \int_{\Vert z\Vert>1}\mathcal{E}z\mu(dt,dz),
\end{equation}
where $\mu(dt,dz)$ is a Poisson random measure on $\R_+\times(\R^d\setminus\{0\})$ with compensator $\nu(dz)dt$. Given that the compensated small-jump integral converges in $L^2$ and the coefficients in \eqref{model3} satisfy local Lipschitz and linear growth conditions, for any $\hat{\mathbf{Q}}^*=(\hat{\mathbf{Q}}_1,\dots,\hat{\mathbf{Q}}_L)\in(\mathcal{M}_K(\R))^L$, an estimator of the parameter collection $\mathbf{Q}^* = (\mathbf{Q}_1,\dots,\mathbf{Q}_L)$, and any initial condition $\xi\in\R^{LK}$, existence and uniqueness of a weak solution follow from \cite[Proposition II.2.32]{jacod2013limit}. 

We denote the induced probability measure by $\mathbb{P}^{\hat{\mathbf{Q}}^*}_{\xi}$, and let $\mathbb{P}^{\mathbf{Q}^{(0)}}_{\xi}$ be the corresponding reference measure with $\mathbf{Q}^{(0)}=(\mathbf{0}_K,\dots,\mathbf{0}_K)$. Then for fixed $t>0$ and any $\hat{\mathbf{Q}}^*$, \cite[Proposition 3.1]{lucchese2026estimation} implies that $\mathbb{P}^{\hat{\mathbf{Q}}^*}_{\xi}$ is absolutely continuous with respect to $\mathbb{P}^{\mathbf{Q}^{(0)}}_{\xi}$ on $\mathcal{F}_t$, with Radon–Nikodym derivative
\begin{equation}
\label{likelihood}
\E^{\mathbb{P}^{\mathbf{Q}^{(0)}}_{\xi}}\left[\frac{d\mathbb{P}^{\hat{\mathbf{Q}}^*}_{\xi}}{d\mathbb{P}^{\mathbf{Q}^{(0)}}_{\xi}}\biggr\vert\mathcal{F}_t\right]=\exp\left\lbrace-\int_0^t\mathbf{X}_{s-}^{\top}\hat{\mathbb{Q}}^{*{\top}}\tilde{\Sigma}^{-1}d\mathbf{X}_{\mathbf{Q}^{(0)},\xi,s}^{c}-\frac{1}{2}\int_0^t\mathbf{X}_s^{\top}\hat{\mathbb{Q}}^{*{\top}}\tilde{\Sigma}^{-1}\hat{\mathbb{Q}}^*\mathbf{X}_sds\right\rbrace.
\end{equation}
Here, $\mathbf{X}_{\mathbf{Q}^{(0)},\xi,t}^{c}$ denotes the continuous local martingale part
\begin{equation}
\label{martingale}
\mathbf{X}_{\mathbf{Q}^{(0)},\xi,t}^{c}=\mathbf{X}_t-\xi-\mathcal{E}\mathbf{b}t-\mathbb{Q}^{(0)}\int_0^t\mathbf{X}_sds
-\sum_{s\leq t}\Delta \mathbf{X}_s\mathbbm{1}_{\Vert\Delta\mathbf{X}_s\Vert>1}-\int_0^t\int_{\Vert x\Vert\leq 1}x(\mu_\mathbf{X}(dx,dt)-\tilde{\nu}(dx)dt),
\end{equation}
where the finite jump component $\Delta\mathbf{X}_s := \mathbf{X}_s-\mathbf{X}_{s-}$ and $\tilde{\nu}(dx)=\int_{\R^K}\mathbbm{1}_{\mathcal{E}z=dx}\nu(dz)$ is the  L\'evy measure of $\mathbf{X}$. The matrices $\hat{\mathbb{Q}}^*,\mathbb{Q}^{(0)}$ and $\tilde{\Sigma}^{-1}$ take the block forms
\begin{equation*}
    \hat{\mathbb{Q}}^*=\begin{bmatrix}
    \mathbf{0}_{K} & \mathbf{0}_{K} & \mathbf{0}_{K} & \dots & \mathbf{0}_{K}\\
    \mathbf{0}_{K} & \mathbf{0}_{K} & \mathbf{0}_{K}  & \dots & \mathbf{0}_{K}\\
    \vdots & \vdots & \vdots & \ddots & \vdots  \\
    \mathbf{0}_{K} & \mathbf{0}_{K} & \mathbf{0}_{K} & \dots & \mathbf{0}_{K}\\
    -\hat{\mathbf{Q}}_L & -\hat{\mathbf{Q}}_{L-1} & -\hat{\mathbf{Q}}_{L-2} & \dots  & -\hat{\mathbf{Q}}_1
\end{bmatrix}, \mathbb{Q}^{(0)} = \begin{bmatrix}
    \mathbf{0}_{K} & \mathbf{I}_{K} & \mathbf{0}_{K} & \dots & \mathbf{0}_{K}\\
    \mathbf{0}_{K} & \mathbf{0}_{K} & \mathbf{I}_{K}  & \dots & \mathbf{0}_{K}\\
    \vdots & \vdots & \vdots & \ddots & \vdots  \\
    \mathbf{0}_{K} & \mathbf{0}_{K} & \mathbf{0}_{K} & \dots & \mathbf{I}_{K}\\
    \mathbf{0}_{K} & \mathbf{0}_{K} & \mathbf{0}_{K} & \dots & \mathbf{0}_{K}
\end{bmatrix},
\end{equation*}
and
\begin{equation*}
    \tilde{\Sigma}^{-1}:=\mathcal{E}\Sigma^{-1}\mathcal{E}^{\top}=\begin{bmatrix}
    \mathbf{0}_{K}  &\dots & \mathbf{0}_{K} & \mathbf{0}_{K}\\
    \vdots & \ddots& \vdots  & \vdots  \\
    \mathbf{0}_{K} & \dots & \mathbf{0}_{K} & \mathbf{0}_{K}\\
    \mathbf{0}_{K} & \dots & \mathbf{0}_{K} & \Sigma^{-1}
\end{bmatrix}.
\end{equation*}

Then the likelihood function of $\hat{\mathbf{Q}}^*$ given the observation $\mathbf{Y}_{[0,t]}$ admits the quadratic form
\begin{equation}
\label{likelihoodquad}
\mathcal{L}\left(\hat{\mathbf{Q}}^*;\mathbf{Y}_{[0,t]}\right) := \E^{\mathbb{P}^{\mathbf{Q}^{(0)}}_{\xi}}\left[\frac{d\mathbb{P}^{\hat{\mathbf{Q}}^*}_{\xi}}{d\mathbb{P}^{\mathbf{Q}^{(0)}}_{\xi}}\biggr\vert\mathcal{F}_t\right]=\exp\left\lbrace\hat{\theta}^{*{\top}}\mathbf{K}_t-\frac{1}{2}\hat{\theta}^{*{\top}}[\mathbf{K}]_t\hat{\theta}^{*}\right\rbrace,
\end{equation}
where the linear term in \eqref{likelihood} expands as
\begin{equation*}
    \begin{split}
&\int_0^t\mathbf{X}_{s-}^{\top}\hat{\mathbb{Q}}^{*{\top}}\tilde{\Sigma}^{-1}d\mathbf{X}_{\mathbf{Q}^{(0)},\xi,s}^{c}\\
&=-\int_0^t\left[\underbrace{0,\dots,0}_{(L-1)K},\left(\sum_{l=1}^L\sum_{i=1}^K\mathbf{X}_{s-,l}^{(i)}\hat{\mathbf{Q}}^{\top}_{L-l+1}(i,j)\right)_{j=1,\dots,K}\right]^{\top}\tilde{\Sigma}^{-1}d\mathbf{X}_{\mathbf{Q}^{(0)},\xi,s}^{c}\\
    &=-\int_0^t\left[\underbrace{0,\dots,0}_{(L-1)K},\left(\sum_{j=1}^K\sum_{l=1}^L\sum_{i=1}^K\mathbf{X}_{s-,l}^{(i)}\hat{\mathbf{Q}}^{\top}_{L-l+1}(i,j)\Sigma^{-1}(j,k)\right)_{k=1,\dots,K}\right]^{\top}d\mathbf{X}_{\mathbf{Q}^{(0)},\xi,s}^{c}\\
    &=-\sum_{k=1}^K\int_0^t\sum_{j=1}^K\sum_{l=1}^L\sum_{i=1}^K\mathbf{X}_{s-,l}^{(i)}\hat{\mathbf{Q}}^{\top}_{L-l+1}(i,j)\Sigma^{-1}(j,k)dD^{L-1}\mathbf{Y}^{c,(k)}_{\mathbf{Q}^{(0)},s}\\
    &=-\sum_{k=1}^K\int_0^t\sum_{j=1}^K\sum_{l=1}^L\left(\mathbf{X}_{s-,l}^{(j)}\hat{\alpha}_{e_j,L-l+1}\Sigma^{-1}(j,k)\right.\\
    &\qquad\qquad\left.+\sum_{i=1}^K\sum_{r=1}^{R_l}\mathbf{X}_{s-,l}^{(i)}\hat{\beta}_{L-l+1,r}W^{r}(j,i)\Sigma^{-1}(j,k)\right)dD^{L-1}\mathbf{Y}^{c,(k)}_{\mathbf{Q}^{(0)},s}\\
    &=-\int_0^t\hat{\theta}^{*{\top}}\mathbf{H}_s\Sigma^{-1}dD^{L-1}\mathbf{Y}^{c}_{\mathbf{Q}^{(0)},s}=\hat{\theta}^{*{\top}}\mathbf{K}_t
    \end{split}
\end{equation*}
after reordering indices and using the block structure of $\tilde{\Sigma}^{-1}$. The quadratic term follows analogously from the quadratic covariation of the martingale part. How we obtain the maximum likelihood estimator, weak consistency, and asymptotic normality follows from \cite[Theorem 3.1]{lucchese2026estimation}.
\end{proof}

\section{Discretization scheme and sampling assumptions}
\label{discrete}
Assume that $\{\mathbf{Y}_t,t\geq 0 \}$ is observed on a discrete grid. For each $t>0$, let
\begin{equation*}
    \mathcal{P}_t=\{0=s_0<s_1<\cdots<s_{N_t}=t\},
\end{equation*}
with mesh $\Delta_{\mathcal{P}_t}:=\max\limits_{0\le i<N_t}(s_{i+1}-s_i)$ and observations $\mathbf{Y}_{\mathcal{P}_t}:=\{\mathbf{Y}_s:s\in\mathcal{P}_t\}$. For any process $\mathbf{Z}$, write $\Delta_{\mathcal{P}_t}^n\mathbf{Z}:=\mathbf{Z}_{s_{n+1}}-\mathbf{Z}_{s_n}$.

Our discrete estimator uses
\begin{enumerate}
    \item finite-difference approximations of $D^l\mathbf{Y}_{s_n}$, $l=1,\dots,L-1$;  
    \item Riemann-sum approximations of the stochastic integrals;  
    \item thresholding of increments of $D^{L-1}\mathbf{Y}^c_{\mathbf{Q}^{(0)},s_n}$ to recover the continuous part. 
\end{enumerate}
This requires two grids: a fine grid $\mathcal{P}_t$ for derivative approximation and a coarser grid $\mathcal{Q}_t$ for integral approximation.

\subsection{Derivative approximation}\label{Derivative approximation} \cite[Section 4.2]{lucchese2026estimation}
Define forward finite-difference approximations by
\begin{equation*}
    \hat{D}^l\mathbf{Y}_{\mathcal{P}_t}:=\{\hat{D}^l\mathbf{Y}_{s_n};n=0,\dots,N_t-l\},
\end{equation*}
where 
\begin{equation*}
    \hat{D}^l\mathbf{Y}_{s_n}=\frac{\hat{D}^{l-1}\mathbf{Y}_{s_{n+1}}-\hat{D}^{l-1}\mathbf{Y}_{s_n}}{s_{n+1}-s_n}, \quad l\in\{1,\dots,L-1\}
\end{equation*}
with $\hat{D}^0\mathbf{Y}_{\mathcal{P}_t} = \mathbf{Y}_{\mathcal{P}_t}$. To ensure convergence of higher-order finite differences, we impose a uniformity condition on the observation grid.
\begin{assumption}
\label{assumptionB1}
    The sequence of partitions $\{\mathcal{P}_t;t>0\}$ satisfies, as $t\to\infty$, 
    \begin{equation*}
        \vert c^{l-1}_{\mathcal{P}_t}-1\vert = O\left(\Delta^{l-3/2}_{\mathcal{P}_t}\right)\text{ where } c_{\mathcal{P}_t}:=\frac{\inf\limits_{0\leq i <N_t}s_{i+1}-s_i}{\sup\limits_{0\leq i <N_t}s_{i+1}-s_i}.
    \end{equation*}
\end{assumption}

\subsection{Integral approximation} \label{Integral approximation}\cite[Section 4.1]{lucchese2026estimation}
Let
\begin{equation*}
    \mathcal{Q}_t=\{0=u_0<u_1<\dots<u_{N'_t}=t\}
\end{equation*}
be a coarsening of $\mathcal{P}_t$, so that $N'_t\leq N_t$, $\Delta_{\mathcal{P}_t}\leq\Delta_{\mathcal{Q}_t}$, and $\Delta_{\mathcal{Q}_t}\to 0$ as $t\to\infty$. Define the following jump-related processes from \eqref{martingale}:
\begin{align*}
\mathbf{J}_{u_m}:=\sum_{\substack{u\in\mathcal{Q}_t,u\leq u_m}}\Delta D^{L-1}\mathbf{Y}_u\mathbbm{1}_{\{\Vert \Delta D^{L-1}\mathbf{Y}_u\Vert\geq 1\}},\\
\mathbf{M}_{u_m}:= \int_0^{u_m}\int_{\Vert y \Vert\leq 1}y(\mu_\mathbf{Y}(dy,dt)-\tilde{\nu}(dy)dt).
\end{align*}
The integrals in \eqref{K_t} are approximated by
\begin{equation*}
    \sum_{m=0}^{M_t-1}\hat{\mathbf{H}}_{u_m}\Sigma^{-1}\left(\hat{D}^{L-1}\mathbf{Y}_{\mathbf{Q}^{(0)},u_{m+1}}^{c}-\hat{D}^{L-1}\mathbf{Y}_{\mathbf{Q}^{(0)},u_{m}}^{c}\right)
\end{equation*}
and
\begin{equation*}
    \sum_{m=0}^{M_t-1}\hat{\mathbf{H}}_{u_m}\Sigma^{-1}\hat{\mathbf{H}}_{u_m}^{\top}(u_{m+1}-u_m),
\end{equation*}
where $M_t: = \max\{m\in\{0,\dots,N'_t\}:u_m\leq s_{N_t-L}\}$ and 
\begin{equation}
\label{martingaleY}
    \hat{D}^{L-1}\mathbf{Y}_{\mathbf{Q}^{(0)},u_{m}}^{c} = \hat{D}^{L-1}\mathbf{Y}_{u_m} - \hat{D}^{L-1}\mathbf{Y}_{0}-\mathbf{b}u_m-\mathbf{J}_{u_m}-\mathbf{M}_{u_m}.
\end{equation}

In order to preserve the asymptotic properties of the discrete estimator, the coarse grid must satisfy the usual high-frequency conditions.
\begin{assumption}
\label{assumptionB2}
The partitions $\{\mathcal{Q}_t:t\ge 0\}$ satisfy, as $t\to\infty$,
\begin{equation*}
    t\Delta_{\mathcal{Q}_t}\to 0,\quad c_{\mathcal{Q}_t}\to c>0.
\end{equation*}
\end{assumption}
To prevent degeneracy, we also require
\begin{assumption}
\label{P_Q}
    The partitions $\{\mathcal{P}_t;t>0\}$ and $\{\mathcal{Q}_t;t>0\}$ satisfy $\Delta_{\mathcal{P}_t}=o(t^{-1}\Delta^2_{\mathcal{Q}_t})$ as $t\to\infty$.
\end{assumption}

\subsection{Thresholding and jump activity} \label{Thresholding and jump activity}\cite[Section 4.3]{lucchese2026estimation}
Approximate the increments of the process given in Equation \eqref{martingaleY}, i.e.,
\begin{equation*}
    \Delta^m_{\mathcal{Q}_t}\hat{D}^{L-1}\mathbf{Y}_{\mathbf{Q}^{(0)}}^{c} = \Delta^m_{\mathcal{Q}_t}\hat{D}^{L-1}\mathbf{Y} -\mathbf{b}\Delta^m_{\mathcal{Q}_t}-\Delta^m_{\mathcal{Q}_t}\mathbf{J}-\Delta^m_{\mathcal{Q}_t}\mathbf{M}
\end{equation*}
by the thresholded increments
\begin{equation*}
\Delta^m_{\mathcal{Q}_t}\hat{D}^{L-1}\mathbf{Y}^c_{\mathbf{Q}^{(0)}} :=\left[\Delta^m_{\mathcal{Q}_t}\hat{D}^{L-1}\mathbf{Y}-\mathbf{b}\Delta^m_{\mathcal{Q}_t}\right]\odot\mathbbm{1}_{\left\lbrace \vert \Delta^m_{\mathcal{Q}_t}\hat{D}^{L-1}\mathbf{Y}-\mathbf{b}\Delta^m_{\mathcal{Q}_t}\vert\leq\bm{\nu}^m_t\right\rbrace}.
\end{equation*}

\begin{assumption}
\label{assumptionB4}
Assume the following according to the jump activity of the driving L\'evy process.
\begin{enumerate}
    \item[(a)] \textbf{Finite jump activity}. \cite[4.3.1]{lucchese2026estimation}  
    Equation \eqref{martingaleY} becomes
    \begin{equation*}
\Delta^m_{\mathcal{Q}_t}\hat{D}^{L-1}\mathbf{Y}^c_{\mathbf{Q}^{(0)}}=\Delta^m_{\mathcal{Q}_t}\hat{D}^{L-1}\mathbf{Y}-\mathbf{b}\Delta^m_{\mathcal{Q}_t}-\Delta^m_{\mathcal{Q}_t}\mathbf{J}.
    \end{equation*}
    Let $\lambda^{(i)}$ denote the jump intensity of the $i$-th component and define
\begin{equation*}
\bar{\nu}^{(i)} = \frac{\nu^{(i)}}{\lambda^{(i)}}.
\end{equation*}
Here we impose the condition on the sequence of partitions $\mathcal{Q}_t$ and the thresholding sequences $\bm{\nu}_t^m$ to ensure the asymptotic almost surely for the event that we capture the jump activity if the increment exceeds the threshold for $i\in\{1,\dots,K\}$.
\begin{enumerate}
\item[(i)] $\bm{\nu}_t^{m,(i)}=\left(\Delta^m_{\mathcal{Q}_t}\right)^{\beta^{(i)}}$ with $\beta^{(i)}\in(0,1/2);$
\item[(ii)] $t\Delta_{\mathcal{Q}_t}^{1-2\beta^{(i)}}\to 0, t\to\infty;$
\item[(iii)] $t\bar{\nu}^{(i)}\left(\left(-2\Delta_{\mathcal{Q}_t}^{\beta^{(i)}},2\Delta_{\mathcal{Q}_t}^{\beta^{(i)}}\right)\right)\to 0, t\to\infty.$
\end{enumerate}

    \item[(b)] \textbf{Infinite jump activity}. \cite[4.3.2]{lucchese2026estimation}  
    Equation \eqref{martingaleY} becomes
    \begin{equation*}
        \Delta_{\mathcal{Q}_t}^m \hat{D}^{L-1}\mathbf{Y}_{\mathbf{Q}^{(0)}}^{c}
        =
        \Delta_{\mathcal{Q}_t}^m \hat{D}^{L-1}\mathbf{Y}
        -\mathbf{b}\Delta_{\mathcal{Q}_t}^m
        -\Delta_{\mathcal{Q}_t}^m\mathbf{J}
        -\Delta_{\mathcal{Q}_t}^m\mathbf{M}.
    \end{equation*}
    Define the normalized L\'evy measure $\bar{\nu}:=\nu\vert_{\{\mathbf{x}:\Vert \mathbf{x}\Vert>1\}}/\nu(\{\mathbf{x}:\Vert \mathbf{x}\Vert>1\})$. Then we have the following assumptions on the sequence of partitions $\mathcal{Q}_t$ and the thresholding sequence $\bm{\nu}_t^m$ for $i\in\{1,\dots,K\}$.
    \begin{enumerate}
        \item[(i)] $\bm{\nu}_t^{m,(i)}=\left(\Delta^m_{\mathcal{Q}_t}\right)^{\beta^{(i)}}$ with $\beta^{(i)}\in(0,1/4);$
        \item[(ii)] $t\Delta_{\mathcal{Q}_t}^{1-4\beta^{(i)}}\to 0, t\to\infty;$
        \item[(iii)] $t\bar{\nu}^{(i)}\left(\left(-4\Delta_{\mathcal{Q}_t}^{\beta^{(i)}},4\Delta_{\mathcal{Q}_t}^{\beta^{(i)}}\right)\right)\to 0, t\to\infty;$
        \item[(iv)] $\exists\varepsilon^{(i)}>0$ such that $t\Delta_{\mathcal{Q}_t}^{2\varepsilon^{(i)}}\to 0, t\to\infty$ and 
        \begin{equation*}
            \E_{\mathbf{Q}^*}\left[\vert\mathbf{M}_s^{(i)}\vert1_{\left\lbrace\vert\mathbf{M}_s^{(i)}\vert\leq s^{\beta^{(i)}}\right\rbrace}\right]=O\left(s^{1+\varepsilon^{(i)}}\right) \text{ as } s\downarrow 0;
        \end{equation*}
        \item[(v)] The driving L\'evy process $\mathbf{L}$ has finite fourth-moment.
    \end{enumerate}
\end{enumerate}
\end{assumption}

\section{Proof of Theorem \ref{thm2}}
\label{thm2_proof}
\begin{proof}
    The result is a discretized counterpart of Theorem \ref{thm}. The discretized quantities $\mathbf{K}_{\mathcal{P}_t,\mathcal{Q}_t,\bm{\nu}_t}$ and $[\mathbf{K}]_{\mathcal{P}_t,\mathcal{Q}_t}$ are obtained from $\mathbf{K}_t$ and $[\mathbf{K}]_t$ in Theorem \ref{thm} via Appendix \ref{Derivative approximation}-\ref{Thresholding and jump activity}. Under Assumptions \ref{assumptionB1}-\ref{assumptionB4}, the approximation results of \cite[Section 4 and Appendix D]{lucchese2026estimation} imply that these discretized quantities are asymptotically equivalent to their continuous-time counterparts. Therefore, the discretized estimator $\theta_D^*$ is asymptotically equivalent to the continuous-time estimator $\theta_{MLE}^*$ and complete the proof.
\end{proof}

\section{Simulation data table}
\label{Simulation data table}
Across all three noise levels and network specification scenarios in Tables \ref{sim_table_1}-\ref{sim_table_10}, the predictive accuracy of the competing models is broadly comparable in terms of RMSE, implying that all models are effectively indistinguishable from the naive model according to this metric. Network-based methods such as GNAR and grOU typically perform among the best in terms of directional accuracy, suggesting that explicitly incorporating network structure can improve sign prediction. By contrast, VAR and MCAR tend to be less competitive despite their richer multivariate structure.

\begin{table}[htbp]
\centering
\begin{tabular}{|c|c|c|c|}
\hline
\multicolumn{4}{|c|}{Correct specification}
\\
\hline
Model & RMSE & Directional accuracy & Time(s) \\ \hline
NA   & 0.0498(0.0165) & 0.5 & -\\
AR   & 0.0498(0.0165) & 0.5057(0.0081) & 0.3283(0.0668)\\
VAR  & 0.0501(0.0165) & 0.5020(0.0078) & 1.3022(0.1562)\\
GNAR & 0.0498(0.0165) & 0.5096(0.0084) & 1.1440(0.1952)\\
OU   & 0.0498(0.0165) & 0.5074(0.0082) & 3.8165(0.5432)\\
MCAR & 0.0505(0.0166) & 0.5041(0.0083) & 4.0975(0.4534)\\
grOU & 0.0499(0.0165) & 0.5105(0.0082) & 2.9951(0.3491)\\ \hline

\multicolumn{4}{|c|}{Missing $\alpha = 5$ edge}
\\
\hline
NA   & 0.0498(0.0166) & 0.5 & -\\
AR   & 0.0498(0.0166) & 0.5040(0.0084) & 0.2751(0.0459)\\
VAR  & 0.0501(0.0166) & 0.5013(0.0085) & 1.2135(0.0996)\\
GNAR & 0.0498(0.0167) & 0.5079(0.0086) & 0.8227(0.1246)\\
OU   & 0.0498(0.0167) & 0.5059(0.0087) & 3.1876(0.3717)\\
MCAR & 0.0504(0.0168) & 0.5037(0.0086) & 3.5216(0.3454)\\
grOU & 0.0498(0.0166) & 0.5086(0.0087) & 2.6845(0.2816)\\ \hline

\multicolumn{4}{|c|}{Missing $\alpha = 1$ edge}
\\
\hline
NA   & 0.0499(0.0166) & 0.5 & -\\
AR   & 0.0499(0.0166) & 0.5059(0.0085) & 0.2579(0.0480)\\
VAR  & 0.0502(0.0166) & 0.5020(0.0085) & 1.1470(0.0971)\\
GNAR & 0.0499(0.0166) & 0.5095(0.0086) & 0.7672(0.1198)\\
OU   & 0.0499(0.0166) & 0.5076(0.0085) & 2.9820(0.3271)\\
MCAR & 0.0504(0.0167) & 0.5044(0.0086) & 3.3673(0.3204)\\
grOU & 0.0500(0.0166) & 0.5105(0.0083) & 2.5544(0.2514)\\ \hline
\end{tabular}
\caption{The statistical and computational performance for competing models in 3 scenarios when $\sigma^2=1$.}
\label{sim_table_1}
\end{table}

\begin{table}[htbp]
\centering
\begin{tabular}{|c|c|c|c|}
\hline
\multicolumn{4}{|c|}{Correct specification}
\\
\hline
Model & RMSE & Directional accuracy & Time(s) \\ \hline
NA   & 0.0766(0.0492) & 0.5 & -\\
AR   & 0.0767(0.0491) & 0.5080(0.0108) & 0.2769(0.0427)\\
VAR  & 0.0775(0.0490) & 0.5028(0.0099) & 1.1486(0.0816)\\
GNAR & 0.0766(0.0492) & 0.5152(0.0110) & 0.9950(0.1112)\\
OU   & 0.0766(0.0491) & 0.5129(0.0106) & 3.6324(0.4375)\\
MCAR & 0.0775(0.0495) & 0.5069(0.0099) & 4.2085(0.4883)\\
grOU & 0.0767(0.0491) & 0.5170(0.0114) & 3.1572(0.4242)\\ \hline

\multicolumn{4}{|c|}{Missing $\alpha = 5$ edge}
\\
\hline
NA   & 0.0765(0.0494) & 0.5 & -\\
AR   & 0.0766(0.0494) & 0.5055(0.0108) & 0.2594(0.0406)\\
VAR  & 0.0773(0.0493) & 0.5015(0.0104) & 1.1656(0.0807)\\
GNAR & 0.0765(0.0494) & 0.5131(0.0108) & 0.7945(0.1059)\\
OU   & 0.0765(0.0494) & 0.5117(0.0105) & 3.3502(0.4454)\\
MCAR & 0.0773(0.0497) & 0.5065(0.0102) & 3.7562(0.4647)\\
grOU & 0.0765(0.0494) & 0.5152(0.0112) & 2.9151(0.3977)\\ \hline

\multicolumn{4}{|c|}{Missing $\alpha = 1$ edge}
\\
\hline
NA   & 0.0769(0.0494) & 0.5 & -\\
AR   & 0.0770(0.0494) & 0.5082(0.0112) & 0.2606(0.0439)\\
VAR  & 0.0777(0.0493) & 0.5030(0.0109) & 1.1702(0.0926)\\
GNAR & 0.0769(0.0494) & 0.5147(0.0112) & 0.7799(0.0976)\\
OU   & 0.0769(0.0494) & 0.5132(0.0108) & 3.3827(0.4361)\\
MCAR & 0.0777(0.0497) & 0.5076(0.0109) & 3.8258(0.4687)\\
grOU & 0.0770(0.0494) & 0.5170(0.0115) & 2.9511(0.3869)\\ \hline
\end{tabular}
\caption{The statistical and computational performance for competing models in 3 scenarios when $\sigma^2=5$.}
\label{sim_table_5}
\end{table}

\begin{table}[htbp]
\centering
\begin{tabular}{|c|c|c|c|}
\hline
\multicolumn{4}{|c|}{Correct specification}
\\
\hline
Model & RMSE & Directional accuracy & Time(s) \\ \hline
NA   & 0.0980(0.0749) & 0.5 & -\\
AR   & 0.0981(0.0748) & 0.5096(0.0139) & 0.2880(0.0482)\\
VAR  & 0.0995(0.0745) & 0.5034(0.0123) & 1.1806(0.0877)\\
GNAR & 0.0980(0.0749) & 0.5200(0.0137) & 1.0211(0.1335)\\
OU   & 0.0980(0.0749) & 0.5173(0.0132) & 3.8465(0.4972)\\
MCAR & 0.0992(0.0754) & 0.5098(0.0125) & 4.4385(0.5550)\\
grOU & 0.0980(0.0748) & 0.5221(0.0148) & 3.3500(0.4809)\\ \hline

\multicolumn{4}{|c|}{Missing $\alpha = 5$ edge}
\\
\hline
NA   & 0.0978(0.0752) & 0.5 & -\\
AR   & 0.0980(0.0752) & 0.5066(0.0135) & 0.2554(0.0405)\\
VAR  & 0.0992(0.0749) & 0.5018(0.0123) & 1.1660(0.0780)\\
GNAR & 0.0978(0.0752) & 0.5173(0.0133) & 0.7733(0.0901)\\
OU   & 0.0979(0.0752) & 0.5159(0.0130) & 3.4150(0.4567)\\
MCAR & 0.0988(0.0757) & 0.5097(0.0123) & 3.8597(0.4914)\\
grOU & 0.0978(0.0752) & 0.5207(0.0145) & 3.0209(0.4295)\\ \hline

\multicolumn{4}{|c|}{Missing $\alpha = 1$ edge}
\\
\hline
NA   & 0.0984(0.0752) & 0.5 & -\\
AR   & 0.0985(0.0752) & 0.5099(0.0143) & 0.2610(0.0447)\\
VAR  & 0.0997(0.0749) & 0.5040(0.0133) & 1.1682(0.0942)\\
GNAR & 0.0984(0.0752) & 0.5191(0.0139) & 0.7872(0.1057)\\
OU   & 0.0984(0.0752) & 0.5175(0.0132) & 3.4531(0.4704)\\
MCAR & 0.0993(0.0756) & 0.5107(0.0127) & 3.9494(0.5126)\\
grOU & 0.0985(0.0751) & 0.5219(0.0146) & 3.0789(0.4505)\\ \hline
\end{tabular}
\caption{The statistical and computational performance for competing models in 3 scenarios when $\sigma^2=10$.}
\label{sim_table_10}
\end{table}

\section{The MRC estimator}
\label{mrc}
When working with high-frequency financial data, the standard realized covariance estimator becomes unreliable because the observed transaction prices are contaminated by market microstructure noise. This noise arises from several trading frictions, such as bid–ask bounce, price discreteness, and order processing effects, which cause the recorded price to deviate from the latent efficient price. As a result, when realized covariance is computed using very high-frequency observations, the estimator tends to accumulate noise-induced variation rather than the true quadratic covariation of the underlying price processes.

The Modulated Realized Covariance (MRC) estimator addresses these issues by applying a pre-averaging procedure to the observed price series before computing the covariance matrix. The idea is to smooth short blocks of high-frequency observations so that the noise components, which fluctuate rapidly, tend to cancel out, while the slower-moving efficient price component is preserved. After this filtering step, the covariance of the pre-averaged increments provides a consistent estimator of the integrated covariance of the latent efficient price processes. Consequently, the MRC estimator substantially reduces the impact of microstructure noise while retaining the information contained in high-frequency data, making it a more robust choice than the standard realized covariance estimator.

The code function takes as input a matrix of asset prices $Y$ with $n$ observations and $d$ assets, together with two tuning parameters $\delta$ and $\theta$ that determine the size of the pre-averaging window. Let $Y_i$ denote the $d$-dimensional vector of observed prices at time index $i$. The algorithm first computes a window length
\[
N = \left\lceil (n-1)^{\delta}\theta \right\rceil .
\]
To ensure symmetry in the pre-averaging step, the window size $k_n$ is set equal to $N$ if $N$ is even and $N+1$ otherwise. The window is then divided into two equal halves of length $k_n/2$. Since each pre-averaged increment requires a complete window, the number of valid indices becomes
\[
m = n - k_n + 1 .
\]

For each index $i = 0, \dots, m-1$, the algorithm constructs a pre-averaged return vector $\bar{Y}_i$. This is computed as the difference between the sum of the forward half-window and the sum of the backward half-window, scaled by the window length:
\[
\bar{Y}_i =
\frac{1}{k_n}
\left(
\sum_{j=i+k_n/2}^{i+k_n-1} Y_j
-
\sum_{j=i}^{i+k_n/2-1} Y_j
\right).
\]
This operation produces a smoothed increment that reduces the influence of market microstructure noise while retaining the underlying price variation.

After computing all pre-averaged vectors, the estimator forms the covariance matrix by summing the outer products $\bar{Y}_i \bar{Y}_i^\top$ across all valid indices. The result is multiplied by the scaling factor
\[
\frac{n-1}{n-k_n+1} \cdot \frac{12}{k_n},
\]
which corrects for the window length and the reduced number of observations due to pre-averaging. This yields the modulated realized covariance matrix.

If the option \texttt{is\_corr=True} is selected, the covariance matrix is converted into a correlation matrix by dividing each entry by the product of the corresponding standard deviations. Finally, the algorithm extracts the upper-triangular elements of the resulting matrix and reports them as labeled covariance or correlation estimates for each pair of assets.

\subsection*{Acknowledgement}
This work has been supported through the 
EPSRC NeST Programme grant EP/X002195/1.

\bibliographystyle{agsm}
\bibliography{GraphOUedge}

\end{document}